\documentclass[3p,11pt]{elsarticle}
\usepackage{amssymb}
\usepackage{amsmath}
\usepackage{amsthm}

\DeclareMathAlphabet{\mathcal}{OMS}{cmsy}{m}{n}

\makeatletter
\def\ps@pprintTitle{%
 \let\@oddhead\@empty
 \let\@evenhead\@empty
 \def\@oddfoot{\centerline{\thepage}}%
 \let\@evenfoot\@oddfoot}
\makeatother

\newcommand{\bbC}{\mathbb{C}}
\newcommand{\bbF}{\mathbb{F}}
\newcommand{\bbH}{\mathbb{H}}
\newcommand{\bbK}{\mathbb{K}}
\newcommand{\bbR}{\mathbb{R}}
\newcommand{\bbT}{\mathbb{T}}
\newcommand{\bbZ}{\mathbb{Z}}

\newcommand{\bfA}{\mathbf{A}}
\newcommand{\bfC}{\mathbf{C}}
\newcommand{\bfe}{\mathbf{e}}
\newcommand{\bfE}{\mathbf{E}}
\newcommand{\bfF}{\mathbf{F}}
\newcommand{\bfI}{\mathbf{I}}
\newcommand{\bfJ}{\mathbf{J}}
\newcommand{\bfP}{\mathbf{P}}
\newcommand{\bfT}{\mathbf{T}}
\newcommand{\bfx}{\mathbf{x}}
\newcommand{\bfy}{\mathbf{y}}

\newcommand{\bfone}{\boldsymbol{1}}

\newcommand{\bfxi}{\boldsymbol{\xi}}
\newcommand{\bfXi}{\boldsymbol{\Xi}}
\newcommand{\bfphi}{\boldsymbol{\varphi}}
\newcommand{\bfPhi}{\boldsymbol{\Phi}}
\newcommand{\bfpsi}{\boldsymbol{\psi}}
\newcommand{\bfPsi}{\boldsymbol{\Psi}}
\newcommand{\bfchi}{\boldsymbol{\chi}}
\newcommand{\bfzeta}{\boldsymbol{\zeta}}
\newcommand{\bfdelta}{\boldsymbol{\delta}}

\newcommand{\calA}{\mathcal{A}}
\newcommand{\calB}{\mathcal{B}}
\newcommand{\calD}{\mathcal{D}}
\newcommand{\calE}{\mathcal{E}}
\newcommand{\calG}{\mathcal{G}}
\newcommand{\calH}{\mathcal{H}}
\newcommand{\calK}{\mathcal{K}}
\newcommand{\calM}{\mathcal{M}}
\newcommand{\calN}{\mathcal{N}}
\newcommand{\calS}{\mathcal{S}}
\newcommand{\calU}{\mathcal{U}}
\newcommand{\calV}{\mathcal{V}}

\newcommand{\rmc}{\mathrm{c}}
\newcommand{\rmC}{\mathrm{C}}
\newcommand{\rme}{\mathrm{e}}
\newcommand{\rmi}{\mathrm{i}}
\newcommand{\rms}{\mathrm{s}}

\newcommand{\tr}{\operatorname{tr}}
\newcommand{\Tr}{\operatorname{Tr}}
\newcommand{\coh}{\operatorname{coh}}
\newcommand{\ETF}{\operatorname{ETF}}
\newcommand{\Fro}{{\operatorname{Fro}}}
\newcommand{\RDS}{\operatorname{RDS}}
\newcommand{\dist}{\operatorname{dist}}
\newcommand{\Span}{\operatorname{span}}

\newcommand{\abs}[1]{|{#1}|}

\newcommand{\Bigparen}[1]{\Bigl({#1}\Bigr)}

\newcommand{\bigbracket}[1]{\bigl[{#1}\bigr]}

\newcommand{\set}[1]{\{{#1}\}}
\newcommand{\bigset}[1]{\bigl\{{#1}\bigr\}}

\newcommand{\norm}[1]{\|{#1}\|}

\newcommand{\Bignorm}[1]{\Bigl\|{#1}\Bigr\|}

\newcommand{\ip}[2]{\langle{#1},{#2}\rangle}

\newtheorem{theorem}{Theorem}[section]
\newtheorem{lemma}[theorem]{Lemma}

\theoremstyle{definition}

\newtheorem{example}[theorem]{Example}
\newtheorem{definition}[theorem]{Definition}

\begin{document}
\begin{frontmatter}
\title{Harmonic equiangular tight frames comprised of regular simplices}

\author{Matthew Fickus}
\ead{Matthew.Fickus@afit.edu}
\author{Courtney A. Schmitt}

\address{Department of Mathematics and Statistics, Air Force Institute of Technology, Wright-Patterson AFB, OH 45433}

\begin{abstract}
An equiangular tight frame (ETF) is a sequence of unit-norm vectors in a Euclidean space whose coherence achieves equality in the Welch bound,
and thus yields an optimal packing in a projective space.
A regular simplex is a simple type of ETF in which the number of vectors is one more than the dimension of the underlying space.
More sophisticated examples include harmonic ETFs which equate to difference sets in finite abelian groups.
Recently, it was shown that some harmonic ETFs are comprised of regular simplices.
In this paper, we continue the investigation into these special harmonic ETFs.
We begin by characterizing when the subspaces that are spanned by the ETF's regular simplices form an equi-isoclinic tight fusion frame (EITFF), which is a type of optimal packing in a Grassmannian space.
We shall see that every difference set that produces an EITFF in this way also yields a complex circulant conference matrix.
Next, we consider a subclass of these difference sets that can be factored in terms of a smaller difference set and a relative difference set.
It turns out that these relative difference sets lend themselves to a second, related and yet distinct, construction of complex circulant conference matrices.
Finally, we provide explicit infinite families of ETFs to which this theory applies.
\end{abstract}

\end{frontmatter}

%%%%%%%%%%%%%%%%%%%%%%%%%%%%%%%%%%%%%%%%%%%%%%%%%%%%%%%%%%%%%%%%
\section{Introduction}
%%%%%%%%%%%%%%%%%%%%%%%%%%%%%%%%%%%%%%%%%%%%%%%%%%%%%%%%%%%%%%%%

Let $\bbH$ be a $D$-dimensional complex Hilbert space whose inner product is conjugate-linear in its first argument,
and let $\calN$ be an $N$-element indexing set.
The \textit{Welch bound}~\cite{Welch74} is a lower bound on the \textit{coherence} of a sequence $\set{\bfphi_n}_{n\in\calN}$ of nonzero vectors in $\bbH$:
\begin{equation}
\label{eq.Welch bound}
\smash{\bigbracket{\tfrac{N-D}{D(N-1)}}^{\frac12}
\leq\coh(\{\bfphi_n\}_{n\in\calN})
:=\max_{n\not=n'}
\tfrac{\abs{\ip{\bfphi_n}{\bfphi_{\smash {n'}}}}}{\norm{\bfphi_n}\norm{\bfphi_{\smash{n'}}}}}.
\end{equation}
It is well known~\cite{StrohmerH03} that unit norm vectors $\set{\bfphi_n}_{n\in\calN}$ in $\bbH$ achieve equality in \eqref{eq.Welch bound} if and only if they form an \textit{equiangular tight frame} (ETF) for $\bbH$,
namely when there exists $C>0$ such that
$C\norm{\bfx}^2=\sum_{n\in\calN}\abs{\ip{\bfphi_n}{\bfx}}^2$ for all $\bfx\in\bbH$ (tightness)
and $\abs{\ip{\bfphi_n}{\bfphi_{n'}}}$ is constant over all $n\neq n'$ (equiangularity).
In particular, the lines spanned by an ETF's vectors have the property that the minimum angle between any pair of them is as large as possible,
and so are an optimal packing of points in projective space.
Because of this optimality,
ETFs arise in various applications including waveform design for wireless communication~\cite{StrohmerH03},
compressed sensing~\cite{BajwaCM12,BandeiraFMW13},
quantum information theory~\cite{Zauner99,RenesBSC04} and algebraic coding theory~\cite{JasperMF14}.

ETFs are tricky to construct~\cite{FickusM16}.
To elaborate, letting ``$\ETF(D,N)$" denote an $N$-vector ETF for a $D$-dimensional space $\bbH$,
$\ETF(D,D)$ and $\ETF(D,D+1)$ correspond to orthonormal bases and regular simplices for $\bbH$, respectively, and so exist for every $D$.
Apart from these trivial examples,
every other known infinite family of ETFs arises from some type of combinatorial design.
Real ETFs in particular are equivalent to a subclass of strongly regular graphs~\cite{vanLintS66,Seidel76,HolmesP04,Waldron09},
and such graphs are well studied~\cite{Brouwer07,Brouwer17,CorneilM91}.
This equivalence has been partially generalized to the complex setting in various ways,
including approaches that exploit properties of roots of unity~\cite{BodmannPT09,BodmannE10},
abelian distance-regular covers of complete graphs~\cite{CoutinkhoGSZ16,FickusJMPW19},
and association schemes~\cite{IversonJM16}.
Infinite families of ETFs whose \textit{redundancy} $\frac ND$ is either nearly or exactly two arise from the related concepts of conference matrices, Hadamard matrices, Paley tournaments and Gauss sums~\cite{StrohmerH03,HolmesP04,Renes07,Strohmer08}.
Other constructions are more flexible, allowing one to prescribe the order of magnitude of $D$ and $\frac ND$ almost independently,
including harmonic ETFs and Steiner ETFs.
As detailed in the next section,
harmonic ETFs are equivalent to difference sets in finite abelian groups~\cite{Turyn65,StrohmerH03,XiaZG05,DingF07}.
Meanwhile, Steiner ETFs arise from balanced incomplete block designs~\cite{GoethalsS70,FickusMT12}.
This construction has recently been generalized to yield new infinite families of ETFs arising from projective planes that contain hyperovals,
Steiner triple systems, and group divisible designs~\cite{FickusMJ16,FickusJMP18,FickusJ19}.

By construction, a Steiner ETF is \textit{comprised of regular simplices} in the sense that its vectors can be partitioned into subsequences, each of which is a regular simplex for its span.
Every harmonic ETF arising from a McFarland difference set is known to be unitarily equivalent to a Steiner ETF, and so also has this structure~\cite{JasperMF14}.
In a recent paper~\cite{FickusJKM18}, it was shown that other harmonic ETFs,
including those arising from the complements of certain Singer and twin prime power difference sets, are comprised of regular simplices despite not being unitarily equivalent to any Steiner ETF.
There, it was further shown that when an ETF is comprised of regular simplices,
the subspaces spanned by these simplices form a particular type of optimal packing in Grassmannian space known as an \textit{equi-chordal tight fusion frame} (ECTFF), achieving the simplex bound of~\cite{ConwayHS96}.

Here, we continue this investigation into harmonic ETFs that are comprised of regular simplices.
In the next section, we establish notation and review known concepts that we shall use later on.
In Section~3, we better characterize the properties of difference sets that lead to ETFs comprised of regular simplices; see Theorem~\ref{thm.fine}.
In Theorem~\ref{thm.EITFF}, we then characterize when the subspaces spanned by these simplices form a special type of ECTFF known as an \textit{equi-isoclinic tight fusion frame} (EITFF).
This occurs for some, but not all, of the ETFs considered in~\cite{FickusJKM18}.
We further show that every difference set that produces an EITFF in this way also yields a complex \textit{circulant conference matrix} $\bfC$,
namely an $(S+1)\times(S+1)$ circulant matrix whose diagonal entries are zero,
whose off-diagonal entries are unimodular and for which $\bfC^*\bfC=S\bfI$.
In Section~4, we refine this analysis further,
showing in Theorem~\ref{thm.composite} that a special class of these ETFs arise from difference sets that are a Gordon-Mills-Welch sum of a relative difference set and a smaller difference set.
We further show in Theorem~\ref{thm.simplicial} that these resulting relative difference sets yield collections of regular simplices that are \textit{mutually unbiased} in the quantum-information-theoretic sense,
as well as complex circulant conference matrices in a way that is related to, but distinct from, the method of Section~3.
We then show that two known families of difference sets yield ETFs with these extraordinary properties,
namely the complements of certain Singer difference sets (Theorem~\ref{thm.Singer}),
and the complements of certain twin prime power difference sets (Theorem~\ref{thm.TPP}).
Overall, these two methods yield $(S+1)\times(S+1)$ circulant conference matrices when either $S=Q+1$ where $Q$ is a prime power or $S=Q+2$ where $Q$ and $Q+2$ are twin prime powers with $Q\equiv 3\bmod 4$.

%%%%%%%%%%%%%%%%%%%%%%%%%%%%%%%%%%%%%%%%%%%%%%%%%%%%%%%%%%%%%%%%
\section{Background}
%%%%%%%%%%%%%%%%%%%%%%%%%%%%%%%%%%%%%%%%%%%%%%%%%%%%%%%%%%%%%%%%

Let $z^*$ be the complex conjugate of $z\in\bbC$.
More generally, let $\bfA^*$ denote the adjoint of an operator $\bfA$ between two complex Hilbert spaces.
For any $N$-element indexing set $\calN$,
let $\ip{\bfy_1}{\bfy_2}:=\sum_{n\in\calN}[\bfy_1(n)]^*\bfy_2(n)$ be the standard inner product on  $\bbC^\calN:=\{\bfy:\calN\to\bbC\}$.
For any $M$-element indexing set $\calM$,
we can regard a linear operator from $\bbC^\calN$ to $\bbC^\calM$ as a matrix whose entries are indexed by $\calM\times\calN$, namely as a member of $\bbC^{\calM\times\calN}:=\{\bfA:\calM\times\calN\to\bbC\}$,
a space we equip with the Frobenius (Hilbert-Schmidt) inner product, $\ip{\bfA_1}{\bfA_2}_\Fro:=\Tr(\bfA_1^*\bfA_2^{})$.

The \textit{synthesis operator} of a sequence of vectors $\{\bfphi_n\}_{n\in\calN}$ in Hilbert space $\bbH$ is $\bfPhi:\bbC^\calN\to\bbH$,
$\bfPhi\bfy:=\sum_{n\in\calN}\bfy(n)\bfphi_n$.
Its adjoint is the \textit{analysis operator} $\bfPhi^*:\bbH\to\bbC^\calN$, $(\bfPhi^*\bfx)(n)=\ip{\bfphi_n}{\bfx}$.
In the special case where $\bbH=\bbC^\calM$,
$\bfPhi$ is the $\calM\times\calN$ matrix whose $n$th column is $\bfphi_n$,
and $\bfPhi^*$ is its $\calN\times\calM$ conjugate-transpose.
Composing these operators yields the \textit{frame operator} $\bfPhi\bfPhi^*:\bbH\to\bbH$, $\bfPhi\bfPhi^*\bfx=\sum_{n\in\calN}\ip{\bfphi_n}{\bfx}\bfphi_n$ and the $\calN\times\calN$ \textit{Gram matrix} $\bfPhi^*\bfPhi:\bbC^\calN\to\bbC^\calN$ whose $(n,n')$th entry is $(\bfPhi^*\bfPhi)(n,n')=\ip{\bfphi_n}{\bfphi_{n'}}$.
We sometimes also regard each vector $\bfphi_n$ as a degenerate synthesis operator $\bfphi_n:\bbC\to\bbH$, $\bfphi_n(y)=y\bfphi_n$,
an operator whose adjoint is the linear functional $\bfphi_n^*:\bbH\to\bbC$, $\bfphi_n^*\bfx=\ip{\bfphi_n}{\bfx}$.
Under this notation, the frame operator of $\{\bfphi_n\}_{n\in\calN}$ is
$\bfPhi\bfPhi^*=\sum_{n\in\calN}\bfphi_n^{}\bfphi_n^*$.

We say that $\{\bfphi_n\}_{n\in\calN}$ is a \textit{($C$-)tight frame} for $\bbH$ when $\bfPhi\bfPhi^*=C\bfI$ for some $C>0$.
In this case,
when the vectors $\{\bfphi_n\}_{n\in\calN}$ are regarded as members of some (larger) Hilbert space $\bbK$ which contains $\bbH=\Span\{\bfphi_n\}_{n\in\calN}$ as a (proper) subspace,
we say that $\{\bfphi_n\}_{n\in\calN}$ is a \textit{tight frame for its span}; elsewhere in the literature, such sequences are sometimes called ``tight frame sequences."
Here the analysis operator $\bfPhi^*:\bbH\to\bbC^\calN$ extends to an operator $\bfPhi^*:\bbK\to\bbC^\calN$ and $\{\bfphi_n\}_{n\in\calN}$ is a tight frame for its span precisely when $\bfPhi\bfPhi^*\bfx=C\bfx$ for all $\bfx\in\bbH=\Span(\{\bfphi_n\}_{n\in\calN})=\rmC(\bfPhi)$.
As shown in~\cite{FickusMJ16}, this is equivalent to having either $\bfPhi\bfPhi^*\bfPhi=C\bfPhi$, $(\bfPhi\bfPhi^*)^2=C\bfPhi\bfPhi^*$ or $(\bfPhi^*\bfPhi)^2=C\bfPhi^*\bfPhi$.
In particular, $\{\bfphi_n\}_{n\in\calN}$ is a $C$-tight frame for some $D$-dimensional space if and only if its Gram matrix $\bfPhi^*\bfPhi$ has eigenvalues $C$ and $0$ with multiplicity $D$ and $N-D$, respectively.

A \textit{Naimark complement} of an $N$-vector $C$-tight frame $\{\bfphi_n\}_{n\in\calN}$ for a $D$-dimensional space $\bbH$ is any sequence $\{\bfpsi_n\}_{n\in\calN}$ of vectors in some space $\bbK$ such that $\bfPhi^*\bfPhi+\bfPsi^*\bfPsi=C\bfI$.
Since $\bfPsi^*\bfPsi$ has eigenvalues $C$ and $0$ with multiplicity $N-D$ and $D$,
respectively, $\{\bfpsi_n\}_{n\in\calN}$ is a $C$-tight frame for its $(N-D)$-dimensional span.
Being defined in terms of Gram matrices,
Naimark complements are unique up to unitary transformations.
They exist whenever $N>D$:
one way to construct one is to regard $\bbH$ as $\bbC^D$,
and take $\set{\bfpsi_n}_{n\in\calN}$ to be the columns of the $(N-D)\times\calN$ matrix $\bfPsi$ whose rows, when taken together with the rows of $\bfPhi$, form an equal-norm orthogonal basis for $\bbC^N$.

\subsection{Equi-chordal and equi-isoclinic tight fusion frames}

When $\set{\bfphi_n}_{n\in\calN}$ is a sequence of unit norm vectors,
its frame operator $\bfPhi\bfPhi^*=\sum_{n\in\calN}\bfphi_n^{}\bfphi_n^*$ is the sum of the orthogonal projection operators onto their $1$-dimensional spans.
More generally, if $\{\calU_n\}_{n\in\calN}$ is any sequence of $M$-dimensional subspaces of $\bbH$,
its \textit{fusion frame operator} is the sum of the corresponding orthogonal projection operators $\set{\bfP_n}_{n\in\calN}$.
In particular, $\{\calU_n\}_{n\in\calN}$ is a \textit{tight fusion frame} (TFF) for $\bbH$ if there exists $C>0$ such that $C\bfI=\sum_{n\in\calN}\bfP_n$.
Here, the tight fusion frame constant is necessarily $C=\frac{MN}{D}$ since $CD=\Tr(C\bfI)=\sum_{n\in\calN}\Tr(\bfP_n)=MN$.
As such, any sequence $\{\calU_n\}_{n\in\calN}$ of $M$-dimensional subspaces of $\bbH$ satisfies
\begin{equation*}
0
\leq\Bignorm{\sum_{n\in\calN}\bfP_n-\tfrac{MN}{D}\bfI}_\Fro^2
=\sum_{n\in\calN}\sum_{n'\in\calN}\ip{\bfP_n}{\bfP_{n'}}_\Fro-\tfrac{M^2N^2}{D}
=\sum_{n\in\calN}\sum_{n'\neq n}\Tr(\bfP_n\bfP_{n'})
-\tfrac{MN(MN-D)}{D},
\end{equation*}
and achieves equality in this bound if and only if $\{\calU_n\}_{n\in\calN}$ is a TFF for $\bbH$.
At the same time, any such $\{\calU_n\}_{n\in\calN}$ also satisfies $\sum_{n\in\calN}\sum_{n'\neq n}\Tr(\bfP_n\bfP_{n'})
\leq N(N-1)\max_{n\neq n'}\Tr(\bfP_n\bfP_{n'})$,
and achieves equality in this bound if and only if it is \textit{equi-chordal},
namely when the \textit{(squared) chordal distance}
$\dist_{\rmc}^2(\calU_n,\calU_{n'}):=\frac12\norm{\bfP_n-\bfP_{n'}}^2=M-\Tr(\bfP_n\bfP_{n'})$ between any pair of subspaces is the same.
Combining these two inequalities gives that any such $\set{\calU_n}_{n\in\calN}$ satisfies
\begin{equation}
\label{eq.chordal Welch}
\tfrac{M(MN-D)}{D(N-1)}
\leq\max_{n\neq n'}\Tr(\bfP_n\bfP_{n'})
=M-\min_{n\neq n'}\dist_{\rmc}^2(\calU_n,\calU_{n'}),
\end{equation}
and achieves equality in this bound if and only if it is a TFF for $\bbH$ that is also equi-chordal, namely an ECTFF for $\bbH$.
When rewritten as \smash{$\min_{n\neq n'}\dist_{\rmc}^2(\calU_n,\calU_{n'})
\leq \tfrac{M(D-M)N}{D(N-1)}$},
namely as the \textit{simplex bound} of~\cite{ConwayHS96},
we see that an ECTFF $\set{\calU_n}_{n\in\calN}$ has the property that the minimum chordal distance between any pair of these subspaces is as large as possible.
In particular, with respect to the chordal distance, an ECTFF is an optimal packing in the \textit{Grassmannian space} that consists of all $M$-dimensional subspaces of $\bbH$.

Continuing, for any given sequence $\set{\calU_n}_{n\in\calN}$ of $M$-dimensional subspaces of $\bbH$,
we for each $n\in\calN$ let $\bfE_n$ be the synthesis operator of an orthonormal basis $\{\bfe_{n,m}\}_{m\in\calM}$ for $\calU_n$,
and so $\bfE_n^*\bfE_n^{}=\bfI$ and $\bfP_n^{}=\bfE_n^{}\bfE_n^*$.
Here, since
$\sum_{n\in\calN}\bfP_n
=\sum_{n\in\calN}\bfE_n^{}\bfE_n^*
=\sum_{n\in\calN}\sum_{m\in\calM}\bfe_{n,m}^{}\bfe_{n,m}^*$,
we have that $\set{\calU_n}_{n\in\calN}$ is a TFF for $\bbH$ if and only if the concatenation (union) $\{\bfe_{n,m}\}_{n\in\calN,m\in\calM}$ of these bases is a tight frame for $\bbH$.
Moreover,
\smash{$\Tr(\bfP_n\bfP_{n'})
=\Tr(\bfE_n^*\bfE_n^{}\bfE_{n'}^*\bfE_{n'}^{})
=\norm{\bfE_n^*\bfE_{n'}^{}}_\Fro^2
=\sum_{m=1}^M\sigma_{n,n',m}^2$}
where $\set{\sigma_{n,n',m}}_{m=1}^M$ are the singular values of the $\calM\times\calM$ \textit{cross-Gram} matrix $\bfE_n^*\bfE_{n'}^{}$,
arranged without loss of generality in decreasing order.
Here, since the induced $2$-norm of such a cross-Gram matrix satisfies
$\sigma_{n,n',1}=\norm{\bfE_n^*\bfE_{n'}^{}}_2\leq\norm{\bfE_n}_2\norm{\bfE_{n'}}_2=1$,
there exists an increasing sequence $\set{\theta_{n,n',m}}_{m=1}^M$ in $[0,\frac\pi2]$ such that $\sigma_{n,n',m}=\cos(\theta_{n,n',m})$.
These \textit{principal angles} determine the chordal distances between subspaces:
$\dist_{\rmc}^2(\calU_n,\calU_{n'})
=M-\Tr(\bfP_n\bfP_{n'})
=M-\sum_{m=1}^M\cos^2(\theta_{n,n',m})
=\sum_{m=1}^M\sin^2(\theta_{n,n',m})$.

An EITFF is a special type of ECTFF whose principal angles are constant.
To elaborate,
\smash{$\Tr(\bfP_n\bfP_{n'})
=\sum_{m=1}^M\cos^2(\theta_{n,n',m})
\leq M\cos^2(\theta_{n,n',1})
=M\norm{\bfE_n^*\bfE_{n'}^{}}_2^2$}
for any $n\neq n'$.
Moreover, $\calU_n$ and $\calU_{n'}$ achieve equality here if and only if they are \textit{isoclinic} in the sense that $\set{\theta_{n,n',m}}_{m=1}^{M}$ is constant over $m$.
This happens precisely when, for some $\sigma_{n,n'}\geq0$, we have
$\bfE_n^*\bfE_{n'}^{}\bfE_{n'}^*\bfE_n^{}=\sigma_{n,n'}^2\bfI$,
or equivalently that $\bfP_n\bfP_{n'}\bfP_n=\sigma_{n,n'}^2\bfP_n$.
When combined with~\eqref{eq.chordal Welch}, these facts imply
\begin{equation}
\label{eq.spectral Welch}
\tfrac{MN-D}{D(N-1)}
\leq\tfrac1M\max_{n\neq n'}\Tr(\bfP_n\bfP_{n'})
\leq\max_{n\neq n'}\norm{\bfE_n^*\bfE_{n'}^{}}_2^2
=1-\min_{n\neq n'}\dist_{\rms}^2(\calU_n,\calU_{n'}),
\end{equation}
where $\dist_{\rms}^2(\calU_n,\calU_{n'})
:=1-\norm{\bfE_n^*\bfE_{n'}^{}}_2^2
=\sin^2(\theta_{n,n',1})$ is the \textit{(squared) spectral distance} between $\calU_n$ and $\calU_{n'}$~\cite{DhillonHST08}.
Moreover, $\set{\calU_n}_{n\in\calN}$ achieves equality throughout~\eqref{eq.spectral Welch} if and only if it is an ECTFF for $\bbH$ where each pair of subspaces is isoclinic,
or equivalently, a TFF for $\bbH$ that is \textit{equi-isoclinic}~\cite{LemmensS73b} in the sense that $\theta_{n,n',m}$ is constant over all $n\neq n'$ and $m$,
namely when there is some $\sigma\geq0$ such that $\bfP_n\bfP_{n'}\bfP_n=\sigma^2\bfP_n$ for all $n\neq n'$.
In particular, every EITFF is an optimal packing in Grassmannian space with respect to both the chordal distance and the spectral distance.

In the special case where $M=1$,
we can let $\bfE_n=\bfphi_n$ be any unit vector in the line $\calU_n$,
both~\eqref{eq.chordal Welch} and \eqref{eq.spectral Welch} reduce to the Welch bound~\eqref{eq.Welch bound},
and $\set{\bfphi_n}_{n\in\calN}$ achieves equality in this bound if and only if it is a tight frame for $\bbH$ that is also equiangular, namely an ETF for $\bbH$.
An $\ETF(D,D)$ equates to $D$ orthonormal vectors.
Meanwhile, when $D<N$, any $\ETF(D,N)$ $\set{\bfphi_n}_{n\in\calN}$ has a Naimark complement $\set{\bfpsi_n}_{n\in\calN}$ which itself is tight and satisfies $\bfPsi^*\bfPsi=C\bfI-\bfPhi^*\bfPhi$,
and so normalizing these vectors yields an $\ETF(N-D,N)$.
In particular, an $\ETF(S,S+1)$ exists for any positive integer $S$,
being equivalent to a Naimark complement of an $\ETF(1,S+1)$,
namely to a sequence of $S+1$ unimodular scalars.
We refer to any $\ETF(S,S+1)$ as a \textit{regular $S$-simplex}.
In light of the Welch bound~\eqref{eq.Welch bound},
any $S+1$ linearly dependent unit vectors with coherence $\frac1S$ necessarily form a regular simplex for their span.

ECTFFs and EITFFs that consist of subspaces of dimension $M\geq 2$ have received some attention in the literature~\cite{LemmensS73b,Hoggar77,ConwayHS96,Zauner99,KutyniokPCL09,BojarovskaP15,King16,EtTaoui18,BlokhuisBE18}
but not nearly as much as ETFs.
EITFFs seem particularly tricky to construct.
One approach is to start with a given EITFF and take a Naimark complement of the tight frame formed by concatenating any orthonormal bases for the subspaces that comprise it.
Doing so converts an EITFF for a $D$-dimensional space that consists of $N$ subspaces of dimension $M$ into an EITFF for an $(MN-D)$-dimensional space that consists of $N$ subspaces of dimension $M$.
Another known method for constructing EITFFs harkens back to~\cite{LemmensS73b}:
if $\set{\bfdelta_m}_{m\in\calM}$ is any orthonormal basis for an $M$-dimensional space $\bbH$, and for each $m\in\calM$, \smash{$\set{\bfphi_n^{(m)}}_{n\in\calN}$} is any $\ETF(D,N)$ for $\bbK$,
then the subspaces
\begin{equation}
\label{eq.ETF tensor ONB}
\set{\calU_n}_{n\in\calN},\quad \calU_n:=\Span\set{\bfdelta_m\otimes\bfphi_n^{(m)}}_{m\in\calM}
\end{equation}
form an EITFF for the $MD$-dimensional space $\bbK\otimes\bbH$ that consists of $N$ subspaces of dimension $M$.
Indeed,
\smash{$\ip{\bfdelta_m\otimes\bfphi_n^{(m)}}{\bfdelta_{m'}\otimes\bfphi_{n'}^{(m')}}
=\ip{\bfdelta_m}{\bfdelta_{m'}}\ip{\bfphi_n^{(m)}}{\bfphi_{n'}^{(m')}}$}
for all $n,n'\in\calN$ and $m,m'\in\calM$,
implying that for each $n\in\calN$, \smash{$\set{\bfdelta_m\otimes\bfphi_n^{(m)}}_{m\in\calM}$} is an orthonormal basis for $\calU_n$,
and moreover that for any $n\neq n'$,
the corresponding cross-Gram matrix $\bfE_n^*\bfE_{n'}^{}$ is diagonal with diagonal entries of modulus \smash{$[\tfrac{N-D}{D(N-1)}]^{\frac12}$}.
As such, $\set{\calU_n}_{n\in\calN}$ achieves equality in~\eqref{eq.spectral Welch} where ``$D$" is $MD$.

Yet another method for constructing EITFFs is to replace each entry of the synthesis operator $\bfPhi$ of a complex $\ETF(D,N)$ with its $2\times 2$ representation as an operator from $\bbR^2\rightarrow\bbR^2$,
namely to apply the one-to-one ring homomorphism
$z=x+\rmi y\mapsto[\begin{smallmatrix}x&-y\\y& x\end{smallmatrix}]$ to every entry of $\bfPhi$;
the resulting pairs of columns form orthonormal bases for $N$ subspaces of $\bbR^{2D}$, each of dimension $M=2$, that form an EITFF for $\bbR^{2D}$~\cite{Hoggar77}.
Recently, it was shown that applying the related yet distinct mapping
$z=x+\rmi y\mapsto[\begin{smallmatrix}x&y\\y&-x\end{smallmatrix}]$
to the entries of a symmetric complex conference matrix of size $N$ yields a matrix which, when suitably scaled and added to an identity matrix,
yields the Gram matrix of an EITFF for $\bbR^N$ that consists of $N$ subspaces of dimension $M=2$~\cite{EtTaoui18}.
Such a conference matrix can be obtained from a real symmetric conference matrix of size $N+1$~\cite{BlokhuisBE18},
which itself equates to a real $\ETF(\frac12(N+1),N+1)$.
For example, the well-known real $\ETF(3,6)$ yields $5$ planes that form an EITFF for $\bbR^5$.
Notably, to date, \cite{EtTaoui18,BlokhuisBE18} give the only known method for constructing EITFFs in which the dimension of the subspaces does not divide the dimension of the space they span, and so are verifiably not of type~\eqref{eq.ETF tensor ONB}.

\subsection{Harmonic equiangular tight frames and difference sets}

A \textit{character} on a finite abelian group $\calG$ is a homomorphism $\gamma:\calG\to\bbT=\{z\in\bbC:\abs{z}=1\}$.
The \textit{(Pontryagin) dual} of $\calG$ is the set $\hat{\calG}$ of all characters of $\calG$,
and is itself a group under pointwise multiplication.
In the general setting, we shall denote the group operations on $\calG$ and $\hat{\calG}$ as addition and multiplication, respectively.
It is well known that since $\calG$ is finite, $\hat{\calG}$ is isomorphic to $\calG$,
and moreover that $\smash{\set{\gamma: \gamma\in\hat{\calG}}}$ is an equal-norm orthogonal basis for $\bbC^\calG$,
meaning its synthesis operator \smash{$\bfF:\bbC^{\hat{\calG}}\rightarrow\bbC^{\calG}$} is invertible with $\bfF^{-1}=\frac1G\bfF^*$ where $G$ is the order of $\calG$.
This operator is usually regarded as the \smash{$(\calG\times\hat{\calG})$}-indexed \textit{character table} of $\calG$ whose $(g,\gamma)$th entry is $\bfF(g,\gamma)=\gamma(g)$.
Its adjoint is the \textit{discrete Fourier transform} (DFT) on $\calG$,
$(\bfF^*\bfx)(\gamma)=\ip{\gamma}{\bfx}$.

Since $\bfF\bfF^*=G\bfI$,
the rows of $\bfF$ are equal-norm orthogonal.
Of course, any subset of these rows also has this property:
if $\calD$ is any nonempty $D$-element subset of $\calG$,
then letting $\bfPhi$ be the $(\calD\times\hat{\calG})$-index defined by \smash{$\bfPhi(d,\gamma)=\frac1{\sqrt{D}}\gamma(d)$},
we have $\bfPhi\bfPhi^*=\frac GD\bfI$.
Regarding the $\gamma$th column of $\bfPhi$ as the unit norm vector \smash{$\bfphi_\gamma=\frac1{\sqrt{D}}\gamma\in\bbC^\calD$},
we equivalently have that \smash{$\set{\bfphi_\gamma}_{\gamma\in\hat{\calG}}$}
is a tight frame for $\bbC^\calD$.
Such tight frames are dubbed \textit{harmonic frames},
and have a circulant Gram matrix with
$\ip{\bfphi_\gamma}{\bfphi_{\gamma'}}
=\tfrac1D\sum_{d\in\calD}(\gamma^{-1}\gamma')(d)
=\tfrac1D(\bfF^*\bfchi_\calD)(\gamma(\gamma')^{-1})$
where $\bfchi_\calD$ is the characteristic function of $\calD$.

The \textit{convolution} of $\bfx_1,\bfx_2\in\bbC^\calG$ is
$\bfx_1*\bfx_2\in\bbC^\calG$,
\smash{$(\bfx_1*\bfx_2)(g):=\sum_{g'\in\calG}\bfx_1(g')\bfx_2(g-g')$},
and satisfies
$[\bfF^*(\bfx_1*\bfx_2)](\gamma)
=(\bfF^*\bfx_1)(\gamma)(\bfF^*\bfx_2)(\gamma)$ for all $\gamma\in\hat{\calG}$.
Meanwhile, the Fourier transform of the  \textit{involution} $\tilde{\bfx}(g):=[\bfx(-g)]^*$ of $\bfx\in\bbC^\calG$ is the pointwise complex conjugate of $\bfF^*\tilde{\bfx}$.
In particular, $[\bfF^*(\bfchi_\calD*\tilde{\bfchi}_\calD)](\gamma)=\abs{(\bfF^*\bfchi_\calD)(\gamma)}^2$ for all $\gamma\in\hat{\calG}$
where, for any $g\in\calG$,
\begin{equation*}
\smash{(\bfchi_\calD*\tilde{\bfchi}_\calD)(g)
=\sum_{g'\in\calD}\bfchi_\calD(g')\bfchi_\calD(g'-g)
=\#[D\cap(g+D)]
=\#\set{(d,d')\in\calD\times\calD: g=d-d'}}
\end{equation*}
is the number of times $g$ can be written as a difference of members of $\calD$.

Now let $\calH$ be any subgroup of $\calG$ of order $H$.
The \textit{Poisson summation formula} states that $\bfF^*\bfchi_\calH=H\bfchi_{\calH^\perp}$ where $\calH^\perp:=\set{\gamma\in\hat{\calG}: \gamma(h)=1,\ \forall\ h\in\calH}$ is the \textit{annihilator} of $\calH$.
It is well known that $\calH^\perp$ is a subgroup of $\hat{\calG}$,
and that $\calH^\perp$ is isomorphic to the dual of $\calG/\calH$,
via the identification of $\gamma\in\calH^\perp$ with the mapping
$\overline{g}\mapsto\gamma(g)$;
here and throughout,
we denote the cosets $g+\calH$ and $\gamma\calH^\perp$ as simply $\overline{g}$ and $\overline{\gamma}$, respectively.
In particular, $\calH^\perp$ has order $\frac GH$.
A subset $\calD$ of $\calG$ is an \textit{$\calH$-relative difference set} (RDS) of $\calG$ if every $g\notin\calH$ can be written as a difference of members of $\calD$ the same number of times, while no nonzero member of the ``forbidden subgroup" $\calH$ can be written in this way,
namely if and only if there exists a scalar $\Lambda$ such that
$\bfchi_\calD*\tilde{\bfchi}_\calD=D\bfdelta_0+\Lambda(\bfone-\bfchi_\calH)$,
where $\bfone=\bfchi_\calG$ is the all-ones vector.
Taking Fourier transforms, this equates to having
\begin{equation*}
\abs{\bfF^*\bfchi_\calD}^2
=\bfF^*[D\bfdelta_0+\Lambda(\bfone-\bfchi_\calH)]
=D\bfone+\Lambda(G\bfdelta_1-H\bfchi_{\calH^\perp}).
\end{equation*}
In particular, we necessarily have
$D^2=\abs{(\bfF^*\bfchi_\calD)(1)}^2=D+\Lambda(G-H)$,
that is, $\Lambda=\frac{D(D-1)}{G-H}$;
alternatively, this follows from the fact that each of the $G-H$ members of $\calH^\rmc$ appears the same number of times in the \textit{difference table} of $\calD$,
namely the $(\calD\times\calD)$-indexed matrix whose $(d,d')$th entry is $d-d'$.
As such, under this assumption,
we have that $\calD$ is an $\calH$-RDS if and only if
\begin{equation}
\label{eq.relative Fourier}
\abs{(\bfF^*\bfchi_\calD)(\gamma)}^2
=\left\{\begin{array}{cl}
D-\Lambda H,&\ \gamma\in\calH^\perp,\ \gamma\neq1,\ \smash{\Lambda=\tfrac{D(D-1)}{G-H}},\\
D,          &\ \gamma\notin\calH^\perp,
\end{array}\right.
\end{equation}
namely if and only if the corresponding harmonic tight frame
$\set{\bfphi_\gamma}_{\gamma\in\hat{\calG}}$ satisfies
\begin{equation}
\label{eq.relative harmonic}
\abs{\ip{\bfphi_\gamma}{\bfphi_{\gamma'}}}
=\tfrac1D\left\{\begin{array}{cl}
\sqrt{D-\Lambda H},&\ \overline{\gamma}=\overline{\gamma}',\ \gamma\neq\gamma',\ \smash{\Lambda=\tfrac{D(D-1)}{G-H}},\\
\sqrt{D},          &\ \overline{\gamma}\neq\overline{\gamma}'.
\end{array}\right.
\end{equation}
Following the literature~\cite{Pott95},
we denote such a set $\calD$ as an $\RDS(\frac{G}{H},H,D,\Lambda)$.

In the particular case where $\calH=\set{0}$,
a relative difference set is simply called a \textit{difference set}.
Here, every $g\neq0$ can be written as a difference of members of $\calD$ in exactly $\Lambda=\frac{D(D-1)}{G-1}$ ways.
Moreover, since \smash{$\set{0}^\perp=\hat{\cal{G}}$} and
\begin{equation*}
D-\Lambda H
=D-\tfrac{D(D-1)}{G-1}
=\tfrac{D(G-D)}{G-1}
=\tfrac{D^2}{S^2},
\quad
S:=[\tfrac{D(G-1)}{G-D}]^{\frac12},
\end{equation*}
we see that~\eqref{eq.relative Fourier} and~\eqref{eq.relative harmonic} reduce to having
\begin{equation}
\label{eq.diff set Fourier}
\abs{(\bfF^*\bfchi_\calD)(\gamma)}=\tfrac{D}{S},\ \forall\ \gamma\neq1,
\quad
\text{i.e.,}
\quad
\abs{\ip{\bfphi_\gamma}{\bfphi_{\gamma'}}}=\tfrac1S,\ \forall\ \gamma\neq\gamma'.
\end{equation}
Since $\frac1S$ also happens to be the Welch bound for $G$ vectors in a $D$-dimensional space,
we obtain that $\calD$ is a difference set for $\calG$ if and only if $\set{\bfphi_\gamma}_{\gamma\in\hat{\calG}}$ is an ETF for $\bbC^\calD$.
When $\calD$ is a difference set,
the quantity \smash{$D-\Lambda=\frac{D^2}{S^2}$} is known as the \textit{order} of $\calD$.
The complement $\calD^\rmc$ of any difference set $\calD$ for $\calG$ is another difference set for $\calG$ of the same order,
and the resulting harmonic ETFs are Naimark complementary.
It is also straightforward to verify that shifting or applying a group automorphism of $\calG$ to a difference set $\calD$ of $\calG$ yields another difference set for $\calG$.

When $\calD$ is an $\calH$-RDS for $\calG$,
then for any subgroup $\calK$ of $\calH$ of order $K$,
applying the quotient map $g\mapsto g+\calK$ to $\calD$ produces an $(\calH/\calK)$-RDS for $\calG/\calK$,
transforming an \smash{$\RDS(\frac GH,H,D,\Lambda)$} into an $\RDS(\frac GH,\frac{H}{K},D,K\Lambda)$.
Indeed, $(d+\calK)-(d'+\calK)=g+\calK$ if and only if $d-d'=g+k$ for some $k\in\calK$.
For $g\in\calK$, this occurs exactly $D$ times, namely when $k=-g$, $d\in\calD$ is arbitrary and $d'=-d$.
Meanwhile, for $g\notin\calH$, this occurs exactly $K\Lambda$ times, namely $\Lambda$ times for each $k\in\calK$.
Finally, for $g\in\calH\backslash\calK$, no such $(d,d')$ exist, regardless of the value of $k$.
In particular, quotienting an $\calH$-RDS $\calD$ by $\calH$ produces a $D$-element difference set $\overline{\calD}=\set{\overline{g}: g\in\calD}$ for $\calG/\calH$,
transforming an \smash{$\RDS(\frac GH,H,D,\Lambda)$} into an \smash{$\RDS(\frac GH,1,D,H\Lambda)$}.

%%%%%%%%%%%%%%%%%%%%%%%%%%%%%%%%%%%%%%%%%%%%%%%%%%%%%%%%%%%%%%%%
\section{Equi-isoclinic subspaces arising from harmonic ETFs}
%%%%%%%%%%%%%%%%%%%%%%%%%%%%%%%%%%%%%%%%%%%%%%%%%%%%%%%%%%%%%%%%

\subsection{Fine difference sets}

It was recently shown that certain harmonic ETFs are comprised of regular simplices~\cite{FickusJKM18}.
To see this from basic principles,
let $\calD$ be a $D$-element difference set in an abelian group $\calG$ of order $G$,
let \smash{$\set{\bfphi_\gamma}_{\gamma\in\hat{\calG}}$}, \smash{$\bfphi_\gamma(d):=\frac1{\sqrt{D}}\gamma(d)$} be the corresponding harmonic ETF for $\bbC^\calD$, and let $S=[\frac{D(G-1)}{G-D}]^{\frac12}$ be the reciprocal of the corresponding Welch bound.
If $S$ is an integer,
then any subsequence \smash{$\set{\bfphi_\gamma}_{\gamma\in\calS}$}
of \smash{$\set{\bfphi_\gamma}_{\gamma\in\hat{\calG}}$} that consists of $S+1$ linearly dependent vectors is a regular simplex for its span:
being linearly dependent, \smash{$\set{\bfphi_\gamma}_{\gamma\in\calS}$} is contained in some $S$-dimensional subspace $\calU$ of $\bbC^\calD$;
at the same time, the coherence of \smash{$\set{\bfphi_\gamma}_{\gamma\in\calS}$} is $\frac1S$,
meaning it achieves the Welch bound for any $S+1$ vectors in $\calU$,
and thus is an $\ETF(S,S+1)$ for $\calU$.
Moreover, if $\calD$ is disjoint from a subgroup $\calH$ of $\calG$,
then the vectors indexed by any coset of its annihilator $\calH^\perp$ are trivially dependent, since they sum to zero:
for any $\gamma\in\hat{\calG}$ and $d\in\calD$,
the Poisson summation formula gives
\begin{equation}
\label{eq.simplex sum}
\sum_{\gamma'\in\gamma\calH^\perp}\bfphi_{\gamma'}(d)
=\sum_{\gamma'\in\calH^\perp}\tfrac1{\sqrt{D}}(\gamma\gamma')(d)
=\tfrac1{\sqrt{D}}\gamma(d)(\bfF\bfchi_{\calH^\perp})(d)
=\tfrac1{\sqrt{D}}\gamma(d)\tfrac{G}{H}\bfchi_\calH(d)
=0.
\end{equation}
Altogether,
we see that every subsequence \smash{$\set{\bfphi_{\gamma'}}_{\gamma'\in\gamma\calH^\perp}$}
of the ETF \smash{$\set{\bfphi_\gamma}_{\gamma\in\hat{\calG}}$}
is a regular simplex provided the underlying difference set $\calD$ is disjoint from a subgroup $\calH$ of order \smash{$H=\frac{G}{S+1}$}, that is, whose annihilator $\calH^\perp$ has order $S+1$.
In~\cite{FickusJKM18}, it was further shown that several known families of difference sets have this property.
To make these concepts easier to discuss moving forward, we now give this property a name:

\begin{definition}
\label{def.fine}
A $D$-element difference set $\calD$ for an abelian group $\calG$ of order $G$ is \textit{fine} if there exists a subgroup $\calH$ of $\calG$ that is disjoint from $\calD$ and is of order
$H=\frac{G}{S+1}$ where $S=[\frac{D(G-1)}{G-D}]^{\frac12}$.
When this occurs with a specific $\calH$,
we say $\calD$ is \textit{$\calH$-fine}.
\end{definition}

Below we show that such difference sets $\calD$ are ``fine" in the sense that they can ``pass through a fine sieve," that is, are disjoint from a subgroup $\calH$ of $\calG$ that is as large as any such subgroup can be.
We further show that when a difference set is fine, every nonidentity coset of the corresponding subgroup $\calH$ intersects $\calD$ in the same number of points.
Here, it helps to introduce the following notation:
if $\calH$ is any subgroup of a finite abelian group $\calG$, and $\calD$ is any subset of $\calG$, let
\begin{equation}
\label{eq.def of D_g}
\calD_g:=\calH\cap(\calD-g),\quad \forall\ g\in\calG.
\end{equation}

\begin{example}
\label{ex.8x15 fine}
As a simple example of a complement of a Singer difference set,
let $\calG$ be the cyclic group $\bbZ_{15}$ and let $\calD$ be $\set{6,11,7,12,13,3,9,14}$;
the rationale behind this unusual ordering of the elements of $\calD$ will eventually become apparent.
Computing the difference table of $\calD$,
we see that each of the $14$ nonzero elements of $\calG$ can be written as a difference of members of $\calD$ in exactly $\Lambda=\frac{D(D-1)}{G-1}=\frac{8(7)}{14}=4$ ways:
\begin{equation*}
\begin{array}{c|cccccccc}
- & 6&11& 7&12&13& 3& 9&14\\\hline
 6& 0&10&14& 9& 8& 3&12& 7\\
11& 5& 0& 4&14&13& 8& 2&12\\
 7& 1&11& 0&10& 9& 4&13& 8\\
12& 6& 1& 5& 0&14& 9& 3&13\\
13& 7& 2& 6& 1& 0&10& 4&14\\
 3&12& 7&11& 6& 5& 0& 9& 4\\
 9& 3&13& 2&12&11& 6& 0&10\\
14& 8& 3& 7& 2& 1&11& 5& 0\\
\end{array}.
\end{equation*}
Thus, $\calD$ is a difference set for $\calG$.
Here, \smash{$S=[\frac{D(G-1)}{G-D}]^{\frac12}=[\frac{8(14)}{7}]^{\frac12}=4$} is an integer.
Moreover, $S+1=5$ divides $G=15$, and so the cyclic group $\calG$ contains a unique subgroup of order $H=\frac{G}{S+1}=3$, namely $\calH=\set{0,5,10}$.
Being disjoint from $\calH$, we see that $\calD$ is fine in the sense of Definition~\ref{def.fine}.
Moreover, since the cosets of $\calH$ partition $\calG$,
we can partition $\calD$ into its intersections with these cosets,
namely according to congruency modulo $5$:
\begin{equation*}
\calD
=\set{6,11,7,12,13,3,9,14}
=\emptyset\sqcup\set{6,11}\sqcup\set{7,12}\sqcup\set{13,3}\sqcup\set{9,14}.
\end{equation*}
Here, we note that every nontrivial member of this partition has cardinality $\frac{D}{S}=\frac{8}{4}=2$.
Below, we show this is not a coincidence, showing that this property is equivalent to $\calD$ being fine, in general.

For reasons that will eventually become apparent,
we elect to express this partition in terms of subsets of $\calH$ itself,
as opposed to subsets of cosets of $\calH$.
In particular, in the notation of~\eqref{eq.def of D_g}, we have $\calD_0=\calD_5=\calD_{15}=\emptyset$,
$\calD_1=\calD_2=\calD_8=\calD_4=\set{5,10}$,
$\calD_6=\calD_7=\calD_{13}=\calD_9=\set{0,5}$, and
$\calD_{11}=\calD_{12}=\calD_3=\calD_{14}=\set{0,10}$.
Under this notation,
the portion of $\calD$ that lies in the $g$th coset of $\calH$ is
$\calD\cap(\calH+g)=g+[\calH\cap(\calD-g)]=g+\calD_g$.
\end{example}

As evidenced by this example,
the set $g+\calD_g$ only depends on the coset of $\calH$ to which $g$ belongs:
when $\overline{g}=g+\calH$ and $\overline{g}'=g'+\calH$ are equal, we have $g+\calD_g=\calD\cap(g+\calH)=\calD\cap(g'+\calH)=g'+\calD_{g'}$.
On the other hand, $\calD_g$ itself depends on one's choice of coset representative:
when $\overline{g}=\overline{g}'$, we have $\calD_g=(g'-g)+\calD_{g'}$, which is not equal to $\calD_{g'}$ in general.

\begin{theorem}
\label{thm.fine}
If $\calD$ is a difference set for an abelian group $\calG$ of order $G$ and $\calH$ is any subgroup of $\calG$ of order $H$ that is disjoint from $\calD$,
then $H\leq\frac{G}{S+1}$ where $S=[\frac{D(G-1)}{G-D}]^{\frac12}$.
Moreover, the following are equivalent:
\begin{enumerate}
\renewcommand{\labelenumi}{(\roman{enumi})}
\item $H=\frac{G}{S+1}$, that is, $\calD$ is $\calH$-fine;
\item $(\bfF^*\chi_\calD)(\gamma)=-\frac{D}{S}$ for all $\gamma\in\calH^\perp$, $\gamma\neq1$;\smallskip
\item $\#(\calD_g)=\tfrac{D}{S}$ for all $g\not\in\calH$, where $\calD_g:=\calH\cap(\calD-g)$, i.e.,\ $\bfchi_\calD*\bfchi_\calH=\tfrac{D}{S}\bfchi_{\calH^\rmc}$.
\end{enumerate}
As a consequence, if $\calD$ is fine then $S$ is necessarily an integer that divides $D$,
implying the order $D-\Lambda=\frac{D^2}{S^2}$ of $\calD$ is necessarily a perfect square.
\end{theorem}

\begin{proof}
Recall that since $\calD$ is a difference set for $\calG$,
the DFT of its characteristic function satisfies
$(\bfF^*\bfchi_\calD)(1)=D$ and $\abs{(\bfF^*\bfchi_\calD)(\gamma)}=\frac DS$ for all $\gamma\neq1$.
The fact that $\calH$ is disjoint from $\calD$ along with the Poisson summation formula implies
\begin{equation*}
0
=\ip{\chi_\calH}{\chi_\calD}
=\tfrac{1}{G}\ip{\bfF^*\chi_\calH}{\bfF^*\chi_\calD}
=\tfrac{H}{G}\ip{\chi_{\calH^\perp}}{\bfF^*\chi_\calD}
=\tfrac{H}{G}\Bigparen{D+\sum_{\gamma\in\calH^\perp\backslash\set{1}}(\bfF^*\chi_\calD)(\gamma)},
\end{equation*}
where $\calH^\perp$ is the $\frac GH$-element annihilator of $\calH$.
Multiplying by $\frac{GS}{HD}$ and rearranging then gives
\begin{equation}
\label{eq.pf of fine 1}
S=\sum_{\gamma\in\calH^\perp\backslash\set{1}}[-\tfrac{S}{D}(\bfF^*\chi_\calD)(\gamma)],
\end{equation}
namely an expression for $S$ as a sum of $\frac GH-1$ unimodular numbers.
In particular, applying the triangle inequality to \eqref{eq.pf of fine 1} immediately gives $S\leq\tfrac{G}{H}-1$,
namely the claim that $H\leq\frac{G}{S+1}$.

(i $\Leftrightarrow$ ii)
If \smash{$H=\frac{G}{S+1}$},
\eqref{eq.pf of fine 1} expresses $S$ as a sum of $S$ unimodular numbers,
implying each of these numbers is $1$, that is,
$-\tfrac{S}{D}(\bfF^*\chi_\calD)(\gamma)=1$ for all $\gamma\in\calH^\perp$, $\gamma\neq1$, namely (ii).
Conversely, if (ii) holds, then $-\tfrac{S}{D}(\bfF^*\chi_\calD)(\gamma)=1$ for all $\gamma\in\calH^\perp$, $\gamma\neq1$ and so \eqref{eq.pf of fine 1} gives $S=\frac{G}{H}-1$, namely (i).

(ii $\Rightarrow$ iii)
Since $(\bfF^*\bfchi_\calD)(1)=D$, (ii) implies that the pointwise product
of $\bfF^*\bfchi_\calD$ and $\bfchi_{\calH^\perp}$ is
$(\bfF^*\bfchi_\calD)\bfchi_{\calH^\perp}=\frac{D(S+1)}{S}\bfdelta_1-\frac DS\bfchi_{\calH^\perp}$.
Since (ii) implies (i), we can further simplify this product as
$(\bfF^*\bfchi_\calD)\bfchi_{\calH^\perp}
=\frac{DG}{SH}\bfdelta_1-\frac DS\bfchi_{\calH^\perp}
=\frac{D}{S}(\frac{G}{H}\bfdelta_1-\bfchi_{\calH^\perp})$.
The Poisson summation formula then gives
\begin{equation*}
\bfF^*(\bfchi_\calD*\bfchi_\calH)
=H(\bfF^*\bfchi_\calD)\bfchi_{\calH^\perp}
=\tfrac{D}{S}(G\bfdelta_1-H\bfchi_{\calH^\perp})
=\tfrac{D}{S}\bfF^*(\bfone-\bfchi_\calH)
=\tfrac{D}{S}\bfF^*\bfchi_{\calH^\rmc},
\end{equation*}
namely that $\bfchi_\calD*\bfchi_\calH=\tfrac{D}{S}\bfchi_{\calH^\rmc}$.
Since
\begin{equation*}
(\bfchi_\calD*\bfchi_\calH)(g)
=\sum_{g'\in\calG}\bfchi_\calD(g')\bfchi_\calH(g-g')
=\#[\calD\cap(g+\calH)]
=\#[\calH\cap(\calD-g)]
=\#(\calD_g)
\end{equation*}
for any $g\in\calG$, this equates to having $\#(\calD_g)=\frac{D}{S}$ for all $g\notin\calH$.

(iii $\Rightarrow$ i)
The cosets of $\calH$ partition $\calG$,
and so $\calD$ can be partitioned as $\calD=\sqcup_{g\in\calG/\calH}(g+\calD_g)$.
Since $\calD$ is disjoint from $\calH$, $\calD_g=\emptyset$ for the unique coset representative $g$ that lies in $\calH$.
Combining these facts with (iii) gives
$D=\#(\calD)=\sum_{g\in\calG/\calH}\#(\calD_g)=0+(\frac{G}{H}-1)\frac{D}{S}$, namely (i).

For the final conclusions, we now assume $\calD$ is fine, meaning there exists some subgroup $\calH$ of $\calG$ that is disjoint from $\calD$ and satisfies (i)--(iii).
Here, (i) implies $S=\frac{G}{H}-1$ is necessarily an integer,
and (iii) then gives that $S$ divides $D$.
Thus, the order \smash{$D-\Lambda=\frac{D^2}{S^2}$} of $\calD$ is a perfect square.
\end{proof}

For the interested reader, some alternative proofs of these facts are given in Appendix~A.

\subsection{A new representation of harmonic ETFs arising from fine difference sets}

As summarized in Theorem~7.5 of~\cite{FickusJKM18},
several types of difference sets are known to be fine,
including McFarland difference sets,
the complements of twin prime power difference sets,
and the appropriately shifted complements of ``half" of all Singer difference sets.
Moreover, by Theorem~6.2 of~\cite{FickusJKM18}, every ETF that is comprised of regular simplices---including every harmonic ETF arising from a fine difference set---gives rise to an ECTFF.
Below, we show that some of these ECTFFs are EITFFs while others are not.
To do this, it helps to introduce some more notation.

As before, letting \smash{$\set{\bfphi_\gamma}_{\gamma\in\hat{\calG}}$} be the harmonic ETF arising from a difference set $\calD$ that is fine with respect to $\calH$,
we have that for any $\gamma\in\hat{\calG}$,
the subsequence \smash{$\set{\bfphi_{\gamma'}}_{\gamma'\in\gamma\calH^\perp}$} indexed by the $\gamma$th coset of $\calH^\perp$ is a regular simplex for its span.
Here, to facilitate our work below,
we elect to instead index the vectors in every such regular simplex by a common set, namely $\calH^\perp$.
To be precise,
for any $\gamma\in\hat{\calG}$,
let $\bfPhi_\gamma$ be the synthesis operator of \smash{$\set{\bfphi_{\gamma\gamma'}}_{\gamma'\in\calH^\perp}$},
that is,
\begin{equation}
\label{eq.def of Phi_gamma}
\bfPhi_\gamma\in\bbC^{\calD\times\calH^\perp},
\quad
\bfPhi_\gamma(d,\gamma')
:=\bfphi_{\gamma\gamma'}(d)
=\tfrac1{\sqrt{D}}\gamma(d)\gamma'(d).
\end{equation}
This benefit comes at a small price:
though both \smash{$\set{\bfphi_{\gamma'}}_{\gamma'\in\gamma\calH^\perp}$} and \smash{$\set{\bfphi_{\gamma\gamma'}}_{\gamma'\in\calH^\perp}$} depend on $\gamma$,
the former only depends on the coset $\overline{\gamma}=\gamma\calH^\perp$ to which $\gamma$ belongs,
whereas $\bfPhi_\gamma$ is representative dependent.
That said, when $\overline{\gamma}=\overline{\gamma}'$,
we have \smash{$\bfPhi_\gamma=\bfPhi_{\gamma'}\bfT^{\gamma(\gamma')^{-1}}$} where,
for any $\gamma\in\calH^\perp$,
$\bfT^\gamma$ is the ``translation by $\gamma$ operator over $\calH^\perp$",
namely the $(\calH^\perp\times\calH^\perp)$-indexed permutation matrix defined by $\bfT^\gamma(\gamma_1,\gamma_2)=1$ if and only if $\gamma_1^{}\gamma_2^{-1}=\gamma$.
As such, the column space of $\bfPhi_\gamma$ is independent of coset representative.
In particular, the following notation for these subspaces is well defined:
\begin{equation}
\label{eq.def of U_gamma}
\smash{\set{\calU_{\overline{\gamma}}}_{\overline{\gamma}\in\hat{\calG}/\calH^\perp},
\quad
\calU_{\overline{\gamma}}
:=\Span\set{\bfphi_{\gamma\gamma'}}_{\gamma'\in\calH^\perp}
=\rmC(\bfPhi_\gamma).}
\end{equation}

As mentioned above, the results of~\cite{FickusJKM18} imply that if $\calD$ is $\calH$-fine  then~\eqref{eq.def of U_gamma} is an ECTFF for $\bbC^\calD$.
Below, we show that this ECTFF is an EITFF for $\bbC^\calD$ if and only if $\calD_g=\calH\cap(\calD-g)$ is a difference set for $\calH$ for every $g\in\calG$.
This is nontrivial since the techniques of~\cite{FickusJKM18} do not easily generalize to this harder problem.
There, the key idea is that since $\set{\bfphi_{\gamma\gamma'}}_{\gamma'\in\calH^\perp}$ is an $\ETF(S,S+1)$ for $\calU_{\overline{\gamma}}$ where \smash{$S=[\frac{D(G-1)}{G-D}]^{\frac12}$},
the orthogonal projection operator onto $\calU_{\overline{\gamma}}$ can be written as
\smash{$\bfP_{\overline{\gamma}}=\frac{S}{S+1}\bfPhi_\gamma^{}\bfPhi_\gamma^*$}.
Here, since \smash{$\set{\bfphi_\gamma}_{\gamma\in\hat{\calG}}$} is an ETF for $\bbC^\calD$, we also have $\abs{\ip{\bfphi_{\gamma\gamma_1}}{\bfphi_{\gamma\gamma_2}}}=\frac1S$ for any $\gamma_1,\gamma_2\in\calH^\perp$ provided $\overline{\gamma}\neq\overline{\gamma}'$.
Together, these facts imply that for any $\overline{\gamma}\neq\overline{\gamma}'$,
\begin{equation}
\label{eq.fine is ECTFF}
\Tr(\bfP_{\overline{\gamma}}\bfP_{\overline{\gamma}'})
=\tfrac{S^2}{(S+1)^2}\norm{\bfPhi_\gamma^*\bfPhi_{\gamma'}^{}}_\Fro^2
=\tfrac{S^2}{(S+1)^2}\sum_{\gamma_1\in\calH^\perp}\sum_{\gamma_2\in\calH^\perp}
\abs{\ip{\bfphi_{\gamma\gamma_1}}{\bfphi_{\gamma\gamma_2}}}^2
=1.
\end{equation}
This means these $N=H$ subspaces of $\bbC^\calD$ of dimension $M=S$ achieve equality in~\eqref{eq.chordal Welch} and so form an ECTFF for $\bbC^\calD$:
solving for $G$ in \smash{$S^2=\frac{G(D-1)}{G-D}$} gives
\begin{equation}
\label{eq.parameters relationships}
G
=\tfrac{(S^2-1)D}{S^2-D},
\ \text{ i.e., }\
H
=\tfrac{G}{S+1}
=\tfrac{(S-1)D}{S^2-D},
\ \text{ i.e., }\
\tfrac{S(SH-D)}{D(H-1)}=1.
\end{equation}
One could conceivably continue this approach to characterize when~\eqref{eq.def of U_gamma} is an EITFF for $\bbC^\calD$:
having \smash{$\bfP_{\overline{\gamma}}=\frac{S}{S+1}\bfPhi_\gamma^{}\bfPhi_\gamma^*$}, the goal is to characterize when there exists $\sigma^2$ such that $\bfP_{\overline{\gamma}}\bfP_{\overline{\gamma}'}\bfP_{\overline{\gamma}}
=\sigma^2\bfP_{\overline{\gamma}}$ for all $\overline{\gamma}\neq\overline{\gamma}'$.
We did not pursue this approach, and it seems more complicated than our alternative.

We instead construct orthonormal bases for \eqref{eq.def of U_gamma} that permit the singular values of their cross-Gram matrices to be computed explicitly.
To be precise, for any $\gamma\in\hat{\calG}$,
we obtain an isometry $\bfE_\gamma$ so that $\bfPhi_\gamma=\bfE_\gamma\bfPsi$ where $\bfPsi$ is the synthesis operator of a harmonic regular $S$-simplex that naturally arises in this context:
\begin{equation}
\label{eq.def of Psi}
\bfPsi\in\bbC^{(\calG/\calH)\backslash\{\overline{0}\}\times\calH^\perp},
\quad
\bfPsi(\overline{g},\gamma)=\tfrac1{\sqrt{S}}\gamma(g).
\end{equation}
This matrix $\bfPsi$ is well-defined: if $\overline{g}=\overline{g}'$ then $g-g'\in\calH$ implying $\gamma(g)=\gamma(g')\gamma(g-g')=\gamma(g')$ for all $\gamma\in\calH^\perp$.
Moreover, the columns $\set{\bfpsi_\gamma}_{\gamma\in\calH^\perp}$ of $\bfPsi$ form a regular $S$-simplex: if $\gamma\neq\gamma'$,
\begin{equation*}
\ip{\bfpsi_\gamma}{\bfpsi_{\gamma'}}
=\tfrac1S\sum_{\overline{g}\in\calG/\calH\backslash\set{\overline{0}}}(\gamma^{-1}\gamma')(g)
=-\tfrac1S+\tfrac1S\sum_{\overline{g}\in\calG/\calH}(\gamma^{-1}\gamma')(g)
=-\tfrac1S.
\end{equation*}
Here, we have used the fact that
\smash{$\sum_{\overline{g}\in\calG/\calH}\gamma(g)=0$} for all $\gamma\in\calH^\perp$, $\gamma\neq1$,
something which follows from multiplying this equation by $\gamma(g_0)-1\neq0$.
Thus, \smash{$\bfPsi^*\bfPsi=\frac1S[(S+1)\bfI-\bfJ]$}.
This in turn implies $\bfPsi\bfPsi^*=\frac{S+1}{S}\bfI$,
a fact which is also straightforward to prove directly.

It is not surprising that $\bfPhi_\gamma=\bfE_\gamma\bfPsi$ for some isometry $\bfE_\gamma$:
though we do not rely on this fact in our proof below,
the interested reader can use~\eqref{eq.simplex sum} to verify that $\bfPhi_\gamma$ and $\bfPsi$ have the same Gram matrix,
which in turn implies that such an isometry necessarily exists.
What is remarkable however is that this $\bfE_\gamma$ is necessarily simple and sparse:
if $\bfPhi_\gamma=\bfE_\gamma\bfPsi$ then multiplying this equation by $\bfPsi^*$ gives $\bfE_\gamma=\frac{S}{S+1}\bfPhi_\gamma\bfPsi^*$ and so for any $d\in\calD$, $\overline{g}\in\calG/\calH\backslash\set{\overline{0}}$,
\begin{equation*}
\bfE_\gamma(d,\overline{g})
=\tfrac{S}{S+1}\sum_{\gamma'\in\calH^\perp}\bfPhi_\gamma(d,\gamma')[\bfPsi(\overline{g},\gamma')]^*
=\tfrac{\sqrt{S}}{(S+1)\sqrt{D}}\gamma(d)\sum_{\gamma'\in\calH^\perp}\gamma'(d-g),
\end{equation*}
at which point the Poisson summation formula and the fineness of $\calD$ implies
\begin{equation*}
\bfE_\gamma(d,\overline{g})
=\tfrac{\sqrt{S}}{(S+1)\sqrt{D}}\gamma(d)(\bfF\bfchi_{\calH^\perp})(d-g)
=\tfrac{\sqrt{S}}{(S+1)\sqrt{D}}\gamma(d)\tfrac{H}{G}\bfchi_{\calH}(d-g)
=\tfrac{\sqrt{S}}{\sqrt{D}}\left\{\begin{array}{cl}
\gamma(d),&\ \overline{d}=\overline{g},\\
0,&\ \overline{d}\neq\overline{g}.
\end{array}\right.
\end{equation*}
Moreover, if $\bfE_\gamma$ is an isometry, then its columns are unit norm,
giving yet a third way of proving that (i $\Rightarrow$ iii) in Theorem~\ref{thm.fine}.
Below, for the sake of an elementary self-contained proof, we instead take the above expression for $\bfE_\gamma$ as a given, and use Theorem~\ref{thm.fine} to show that it is an isometry that satisfies $\bfPhi_\gamma=\bfE_\gamma\bfPsi$.
Both here and later on, we shall also make use of the following fact:

\begin{lemma}
\label{lem.small difference sets}
Let $\calD$ be a difference set for $\calG$ that is $\calH$-fine; see  Definition~\ref{def.fine}.
Then a $\frac{D}{S}$-element subset $\calB$ of $\calH$ is a difference set for $\calH$ if and only if
\begin{equation*}
\abs{(\bfF^*\bfchi_{\calB})(\gamma)}^2
=\tfrac{D^2}{S^3}[\bfone+(S-1)\bfchi_{\calH^\perp}](\gamma),
\quad
\forall\ \gamma\in\hat{\cal{G}}.
\end{equation*}
Moreover, in this case, $S^3$ necessarily divides $D^2$.
\end{lemma}

\begin{proof}
By~\eqref{eq.parameters relationships},
\smash{$H-1=\tfrac{(S-1)D}{S^2-D}-1=\tfrac{S(D-S)}{S^2-D}$},
which in turn implies \smash{$\tfrac1{H-1}\tfrac{D}{S}(\tfrac{D}{S}-1)=\frac{D}{S}-\frac{D^2}{S^3}$}.
As such, $\calB$ is a difference set for $\calH$ if and only if for every $g\in\calH$,
\smash{$\#\set{(b,b')\in\calB: g=b-b'}
=\frac{D}{S}-\frac{D^2}{S^3}$};
since Theorem~\ref{thm.fine} gives that $S$ divides $D$,
this is only an integer when $S^3$ divides $D^2$.
Moreover, since $\calB$ is a $\frac{D}{S}$-element subset of $\calH$,
this equates to having $\bfchi_\calB\in\bbC^\calG$ satisfy
\begin{equation*}
(\bfchi_\calB*\tilde{\bfchi}_\calB)(g)
=\#\set{(b,b')\in\calB: g=b-b'}
=\left\{\begin{array}{ll}
\frac DS,&\ g=0,\smallskip\\
\frac{D}{S}-\frac{D^2}{S^3},&\ g\in\calH\backslash\set{0},\smallskip\\
0,&\ g\notin\calH,
\end{array}\right.
\end{equation*}
namely to having
$\bfchi_\calB*\tilde{\bfchi}_\calB
=\tfrac{D}{S^3}[D\bfdelta_0+(S^2-D)\bfchi_\calH]$.
Taking Fourier transforms, and again using~\eqref{eq.parameters relationships} and the Poisson summation formula, this equates to our claim:
\begin{equation*}
\abs{(\bfF^*\chi_\calB)(\gamma)}^2
=\tfrac{D}{S^3}[D\bfone+H(S^2-D)\bfchi_{\calH^\perp}](\gamma)
=\tfrac{D^2}{S^3}[\bfone+(S-1)\bfchi_{\calH^\perp}](\gamma),
\quad\forall\ \gamma\in\hat{\calG}.\qedhere
\end{equation*}
\end{proof}

\begin{theorem}
\label{thm.EITFF}
Let $\calD$ be a difference set for $\calG$ that is $\calH$-fine; see Definition~\ref{def.fine}.
Define $\calD_g$, $\bfPhi_\gamma$ and $\bfPsi$ by~\eqref{eq.def of D_g},
\eqref{eq.def of Phi_gamma} and~\eqref{eq.def of Psi}, respectively.
For any $\gamma\in\hat{\calG}$, let
\begin{equation}
\label{eq.def of E_gamma}
\bfE_{\gamma}\in\bbC^{\calD\times(\calG/\calH)\backslash\set{\overline{0}}},
\quad
\bfE_{\gamma}(d,\overline{g})
=\tfrac{\sqrt{S}}{\sqrt{D}}\left\{\begin{array}{cl}
\gamma(d),&\ \overline{d}=\overline{g},\\
0,&\ \overline{d}\neq\overline{g}.
\end{array}\right.
\end{equation}
Then:
\begin{enumerate}
\renewcommand{\labelenumi}{(\alph{enumi})}

\item
For any $\gamma,\gamma'\in\hat{\calG}$,
$\bfE_\gamma^*\bfE_{\gamma'}^{}$ is a diagonal matrix whose diagonal entries are given by
\begin{equation*}
(\bfE_\gamma^*\bfE_{\gamma'}^{})(\overline{g},\overline{g})
=\tfrac{S}{D}\sum_{\substack{d\in\calD\\\overline{d}=\overline{g}}}(\gamma^{-1}\gamma')(d),
\quad
\forall\ \overline{g}\in\calG/\calH\backslash\set{\overline{0}}.
\end{equation*}
Moreover, for any $\gamma\in\hat{\calG}$,
$\bfE_\gamma$ is an isometry, that is, $\bfE_\gamma^*\bfE_\gamma^{}=\bfI$,
and satisfies $\bfPhi_\gamma=\bfE_\gamma\bfPsi$.

\item
The sequence of subspaces \smash{$\set{\calU_{\overline{\gamma}}}_{\overline{\gamma}\in\hat{\calG}/\calH^\perp}$}
given in~\eqref{eq.def of U_gamma} is an EITFF for $\bbC^\calD$ if and only if every $\calD_g$ is a difference set for $\calH$.
In this case, \smash{$S^3$} necessarily divides \smash{$D^2$}.

\item
If every $\calD_g$ is a difference set for $\calH$,
then for every $\gamma\notin\calH^\perp$, the $(\calG/\calH)$-circulant matrix
\begin{equation*}
\bfC_\gamma\in\bbC^{(\calG/\calH)\times(\calG/\calH)},
\quad
\bfC_\gamma(\overline{g},\overline{g}')
=\tfrac{S^{\frac32}}{D}\sum_{\substack{d\in\calD\\\overline{d}=\overline{g}-\overline{g}'}}\gamma(d),
\end{equation*}
is a conference matrix, that is, satisfies $\bfC_\gamma^*\bfC_\gamma^{}=S\bfI$ where
$\abs{\bfC_\gamma(\overline{g},\overline{g}')}
=\left\{\begin{array}{cl}
0,&\ \overline{g}=\overline{g}',\\
1,&\ \overline{g}\neq\overline{g}'.
\end{array}\right.$
\end{enumerate}
\end{theorem}

\begin{proof}
We first prove (a).
For any $\gamma,\gamma'\in\hat{\calG}$ and $\overline{g},\overline{g}'\in\calG/\calH\backslash\set{\overline{0}}$,
\begin{equation*}
(\bfE_\gamma^*\bfE_{\gamma'}^{})(\overline{g},\overline{g}')
=\sum_{d\in\calD}[\bfE_\gamma(d,\overline{g})]^*\bfE_{\gamma'}(d,\overline{g}')
=\tfrac{S}{D}\sum_{\substack{d\in\calD\\\overline{g}=\overline{d}=\overline{g}'}}(\gamma^{-1}\gamma')(d).
\end{equation*}
When $\overline{g}\neq\overline{g}'$, the above sum is empty,
implying \smash{$(\bfE_\gamma^*\bfE_{\gamma'}^{})(\overline{g},\overline{g}')=0$}.
That is, \smash{$\bfE_\gamma^*\bfE_{\gamma'}^{}$} is diagonal.
Moreover, since
$\set{d\in\calD: \overline{d}=\overline{g}}
=\set{d\in\calD: d\in g+\calH}
=\calD\cap(g+\calH)
=g+\calD_g$,
continuing the above equation gives that the $\overline{g}$th diagonal entry of this matrix is
\begin{equation*}
(\bfE_\gamma^*\bfE_{\gamma'}^{})(\overline{g},\overline{g})
=\tfrac{S}{D}\sum_{\substack{d\in\calD\\\overline{d}=\overline{g}}}(\gamma^{-1}\gamma')(d)
=\tfrac{S}{D}\sum_{h\in\calD_g}(\gamma^{-1}\gamma')(g+h)
=\tfrac{S}{D}(\gamma^{-1}\gamma')(g)(\bfF^*\bfchi_{\calD_g})(\gamma(\gamma')^{-1}).
\end{equation*}
In the special case where $\gamma=\gamma'$,
combining these facts with Theorem~\ref{thm.fine}(iii) gives that \smash{$\bfE_\gamma^*\bfE_\gamma^{}$} is a diagonal matrix whose $\overline{g}$th diagonal entry is $(\bfE_\gamma^*\bfE_\gamma^{})(\overline{g},\overline{g})
=\tfrac{S}{D}(\bfF^*\bfchi_{\calD_g})(1)
=\tfrac{S}{D}\#(\calD_g)
=1$.
That is, for any \smash{$\gamma\in\hat{\calG}$} we have $\bfE_\gamma^*\bfE_\gamma^{}=\bfI$.
For the final claim of (a),
note that since $\calD$ is disjoint from $\calH$,
any $d\in\calD$ lies in exactly one nonidentity coset of $\calH$,
implying that for any \smash{$\gamma'\in\calH^\perp$},
\begin{equation*}
(\bfE_\gamma\bfPsi)(d,\gamma')
=\sum_{\overline{g}\in\calG/\calH\backslash\set{\overline{0}}}
\bfE_\gamma(d,\overline{g})\bfPsi(\overline{g},\gamma')
=\tfrac{\sqrt{S}}{\sqrt{D}}\gamma(d)\bfPsi(\overline{d},\gamma')
=\tfrac1{\sqrt{D}}\gamma(d)\gamma'(d)
=\bfPhi_\gamma(d,\gamma').
\end{equation*}

For (b), recall that $\calU_{\overline{\gamma}}$ is the column space $\bfPhi_\gamma$.
By (a), $\calU_{\overline{\gamma}}=\rmC(\bfPhi_\gamma)=\rmC(\bfE_\gamma\bfPsi)
\subseteq\rmC(\bfE_\gamma)$.
Moreover, multiplying $\bfPhi_\gamma=\bfE_\gamma\bfPsi$ by $\bfPsi^*$ gives
$\bfE_\gamma=\frac{S}{S+1}\bfPhi_\gamma\bfPsi^*$ and so $\rmC(\bfE_\gamma)\subseteq\rmC(\bfPhi_\gamma)=\calU_{\overline{\gamma}}$.
Thus, $\rmC(\bfE_\gamma)=\calU_{\overline{\gamma}}$.
When combined with the fact that $\bfE_\gamma^*\bfE_\gamma^{}=\bfI$,
this implies that the columns of $\bfE_\gamma$ form an orthonormal basis for $\calU_{\overline{\gamma}}$.
As discussed in Section~2, we thus have that \smash{$\set{\calU_{\overline{\gamma}}}_{\overline{\gamma}\in\hat{\calG}/\calH^\perp}$},
a sequence of $N=H$ subspaces of $\bbC^\calD$ of dimension $M=S$,
is an EITFF for $\bbC^\calD$ if and only if for all $\overline{\gamma}\neq\overline{\gamma}'$, every singular value of $\bfE_{\gamma}^*\bfE_{\gamma'}^{}$ is equal to \smash{$[\tfrac{MN-D}{D(N-1)}]^{\frac12}=[\tfrac{SH-D}{D(H-1)}]^{\frac12}=\frac1{\sqrt{S}}$},
cf.~\eqref{eq.spectral Welch} and~\eqref{eq.parameters relationships}.
Moreover,
since $\bfE_{\gamma}^*\bfE_{\gamma'}^{}$ is diagonal,
its singular values are the absolute values of its diagonal entries.
Altogether, we have that \smash{$\set{\calU_{\overline{\gamma}}}_{\overline{\gamma}\in\hat{\calG}/\calH^\perp}$} is an EITFF for $\bbC^\calD$ if and only if
\begin{equation}
\label{eq.pf of EITFF 1}
\abs{(\bfE_\gamma^*\bfE_{\gamma'}^{})(\overline{g},\overline{g})}
=\tfrac{S}{D}\abs{(\bfF^*\bfchi_{\calD_g})(\gamma(\gamma')^{-1})}
=\tfrac1{\sqrt{S}},
\quad\forall\ \overline{\gamma}\neq\overline{\gamma}',\ \overline{g}\in\calG/\calH\backslash\set{\overline{0}}.
\end{equation}
More simply, this equates to having
$\abs{(\bfF^*\bfchi_{\calD_g})(\gamma)}^2=\tfrac{D^2}{S^3}$ for all $g\notin\calH$ and $\gamma\notin\calH^\perp$.
Here, combining Theorem~\ref{thm.fine}(iii) with the fact that $\calD_g$ is a subset of $\calH$ gives that for any $g\notin\calH$ and $\gamma\in\calH^\perp$, $(\bfF^*\bfchi_{\calD_g})(\gamma)=\#(\calD_g)=\frac{D}{S}$.
As such, \smash{$\set{\calU_{\overline{\gamma}}}_{\overline{\gamma}\in\hat{\calG}/\calH^\perp}$} is an EITFF for $\bbC^\calD$ if and only if
\begin{equation*}
\abs{(\bfF^*\bfchi_{\calD_g})(\gamma)}^2
=\tfrac{D^2}{S^3}\left\{\begin{array}{cl}
S,&\gamma\in\calH^\perp\\
1,&\gamma\notin\calH^\perp
\end{array}\right\}
=\tfrac{D^2}{S^3}[\bfone+(S-1)\bfchi_{\calH^\perp}](\gamma),
\quad\forall\ \gamma\in\hat{\calG},\ g\notin\calH.
\end{equation*}
By Lemma~\ref{lem.small difference sets}, this occurs if and only if for every $g\notin\calH$ we have that $\calD_g$ is a difference set for $\calH$,
and moreover, this requires $S^3$ to divide $D^2$.
Also, since $\calD$ is disjoint from $\calH$ we have $\calD_g=\emptyset$ for all $g\in\calH$, which is trivially a difference set for $\calH$.
As such, \smash{$\set{\calU_{\overline{\gamma}}}_{\overline{\gamma}\in\hat{\calG}/\calH^\perp}$} is an EITFF for $\bbC^\calD$ if and only if every $\calD_g$ is a difference set for $\calH$.

For (c), we first make a broader observation:
for any difference set $\calD$ and subgroup $\calH$ of a finite abelian group $\calG$,
and any $\gamma\notin\calH^\perp$,
the set of all translations of the vector
\begin{equation*}
\bfx_{\gamma}\in\bbC^{\calG/\calH},
\quad
\bfx_{\gamma}(\overline{g})
:=\tfrac{S}{D}\sum_{\substack{d\in\calD\\\overline{d}=\overline{g}}}\gamma(d),
\end{equation*}
is orthonormal,
that is, $\tilde{\bfx}_\gamma*\bfx_\gamma=\bfdelta_{\overline{0}}$.
To see this, recall the identification of $\calH^\perp$ with the dual of $\calG/\calH$
given in Section~2.
In particular, for any $\gamma\in\hat{\calG}$,
the DFT of $\bfx_\gamma$ at $\gamma'\in\calH^\perp$ is
\begin{equation*}
(\bfF_{\calG/\calH}^*\bfx_\gamma)(\gamma')
=\sum_{\overline{g}\in\calG/\calH}
(\gamma')^{-1}(g)\bfx_\gamma(\overline{g})
=\tfrac{S}{D}\sum_{\overline{g}\in\calG/\calH}
(\gamma')^{-1}(g)\sum_{\substack{d\in\calD\\\overline{d}=\overline{g}}}\gamma(d).
\end{equation*}
To continue, note that since $\gamma'\in\calH^\perp$, we have $(\gamma')^{-1}(g)=(\gamma')^{-1}(d)$
whenever $\overline{d}=\overline{g}$ and so
\begin{equation*}
(\bfF_{\calG/\calH}^*\bfx_\gamma)(\gamma')
=\tfrac{S}{D}\sum_{\overline{g}\in\calG/\calH}
\sum_{\substack{d\in\calD\\\overline{d}=\overline{g}}}((\gamma')^{-1}\gamma)(d)
=\tfrac{S}{D}\sum_{d\in\calD}((\gamma')^{-1}\gamma)(d)
=\tfrac{S}{D}(\bfF^*\bfchi_\calD)(\gamma^{-1}\gamma').
\end{equation*}
In particular, for any $\gamma\notin\calH^\perp$ and $\gamma'\in\calH^\perp$,
\eqref{eq.diff set Fourier} gives
\smash{$\abs{(\bfF_{\calG/\calH}^*\bfx_\gamma)(\gamma')}
=\tfrac{S}{D}\abs{(\bfF^*\bfchi_\calD)(\gamma^{-1}\gamma')}
=1$}.
Thus,
\smash{$\bfF_{\calG/\calH}^*(\tilde{\bfx}_\gamma*\bfx_\gamma)
=\abs{\bfF_{\calG/\calH}^*\bfx_\gamma}^2
=\bfone
=\bfF_{\calG/\calH}^*\bfdelta_{\overline{0}}$},
and so $\tilde{\bfx}_\gamma*\bfx_\gamma=\bfdelta_{\overline{0}}$ as claimed.

As an aside, note that when $\calD$ is fine,
combining this fact with (a) gives an independent, alternative proof of the result from~\cite{FickusJKM18} that
\smash{$\set{\calU_{\overline{\gamma}}}_{\overline{\gamma}\in\hat{\calG}/\calH^\perp}$}
is an ECTFF for $\bbC^{\calD}$, cf.\ \eqref{eq.fine is ECTFF}:
the disjointness of $\calD$ and $\calH$ gives $\bfx_{\gamma^{-1}\gamma'}(\overline{0})=0$,
and so whenever $\overline{\gamma}\neq\overline{\gamma}'$ we have
\begin{equation*}
\norm{\bfE_\gamma^*\bfE_{\gamma'}^{}}_\Fro^2
=\sum_{\overline{g}\in\calG/\calH\backslash\set{\overline{0}}}
\abs{(\bfE_\gamma^*\bfE_{\gamma'}^{})(\overline{g},\overline{g})}^2
=\sum_{\overline{g}\in\calG/\calH}
\abs{\bfx_{\gamma^{-1}\gamma'}(\overline{g})}^2
=\norm{\bfx_{\gamma^{-1}\gamma'}}^2
=1;
\end{equation*}
in light of~\eqref{eq.parameters relationships},
these $N=H$ subspaces of $\bbC^\calD$ of dimension $M=S$ thus achieve equality in~\eqref{eq.chordal Welch}.

In particular,
in the special case where $\calD$ is fine and every $\calD_g$ is a difference set for $\calH$,
for any $\gamma\notin\calH^\perp$ we have $\bfx_\gamma(\overline{0})=0$ while  \eqref{eq.pf of EITFF 1} gives
\smash{$\abs{\bfx_{\gamma}(\overline{g})}
=\abs{(\bfE_1^*\bfE_{\gamma}^{})(\overline{g},\overline{g})}
=\frac1{\sqrt{S}}$}
for any $\overline{g}\neq\overline{0}$.
Since $\bfC_\gamma$ is a $(\calG/\calH)$-circulant matrix with
$\bfC_\gamma(\overline{g},\overline{g}')=\sqrt{S}\bfx_\gamma(\overline{g}-\overline{g}')$,
this implies the diagonal entries of $\bfC_\gamma$ are zero while its off-diagonal entries are unimodular.
Moreover, $\bfC_\gamma^*\bfC_\gamma^{}=S\bfI$ since
\begin{equation*}
(\bfC_\gamma^*\bfC_\gamma^{})(\overline{g},\overline{g}')
=S\sum_{g''\in\calG/\calH}
[\bfx_\gamma(\overline{g}''-\overline{g})]^*
\bfx_\gamma(\overline{g}''-\overline{g}')
=S(\tilde{\bfx}_\gamma*\bfx_\gamma)(\overline{g}-\overline{g}')
=S\bfdelta_0(\overline{g}-\overline{g}').\qedhere
\end{equation*}
\end{proof}

We give the difference sets that satisfy the condition of (b) a name:

\begin{definition}
\label{def.amalgam}
We say a difference set $\calD$ for a finite abelian group $\calG$ is an \textit{amalgam} if it is $\calH$-fine for some subgroup $\calH$ of $\calG$---see Definition~\ref{def.fine}---and moreover for every $g\in\calG$, $\calD_g:=\calH\cap(\calD-g)$ is a difference set for $\calH$.
\end{definition}

We emphasize that the above proofs of Theorems~\ref{thm.fine} and~\ref{thm.EITFF} are self-contained:
though these results were strongly motivated by those of~\cite{FickusJKM18},
our proofs here do not rely on facts from~\cite{FickusJKM18} in any formal way.
In particular, though~\cite{FickusJKM18} implies that $\bfPhi_\gamma$ is the synthesis operator of a regular simplex for its span,
we do not assume this fact in our proofs above.
In fact, we instead provide an alternative proof of this fact,
directly proving $\bfPhi_\gamma=\bfE_\gamma\bfPsi$ where $\bfE_\gamma$ is an isometry and $\bfPsi$ is the synthesis operator of a regular simplex.

\begin{example}
\label{ex.8x15 embeddings}
As a continuation of Example~\ref{ex.8x15 fine},
recall that $\calD=\set{6,11,7,12,13,3,9,14}$ is a difference set for $\calG=\bbZ_{15}$.
To form the corresponding harmonic ETF, we extract the corresponding $8$ rows from the $15\times15$ character table of $\calG$.
Here, we regard $\hat{\calG}$ as $\bbZ_{15}$, identifying $n\in\bbZ_{15}$ with the character $g\mapsto\omega^{ng}$ where \smash{$\omega=\rme^{\frac{2\pi\rmi}{15}}$}.
(As such, throughout this example, we use additive notation on $\hat{\calG}$.)
That is, the columns $\set{\bfphi_n}_{n\in\bbZ_{15}}$ of
\begin{equation*}
\bfPhi=\frac1{\sqrt{8}}
\left[\begin{array}{cccccccccccccccc}
\omega^{ 0}&\omega^{ 6}&\omega^{12}&\omega^{ 3}&\omega^{ 9}&
\omega^{ 0}&\omega^{ 6}&\omega^{12}&\omega^{ 3}&\omega^{ 9}&
\omega^{ 0}&\omega^{ 6}&\omega^{12}&\omega^{ 3}&\omega^{ 9}\\
\omega^{ 0}&\omega^{11}&\omega^{ 7}&\omega^{ 3}&\omega^{14}&
\omega^{10}&\omega^{ 6}&\omega^{ 2}&\omega^{13}&\omega^{ 9}&
\omega^{ 5}&\omega^{ 1}&\omega^{12}&\omega^{ 8}&\omega^{ 4}\\
\omega^{ 0}&\omega^{ 7}&\omega^{14}&\omega^{ 6}&\omega^{13}&
\omega^{ 5}&\omega^{12}&\omega^{ 4}&\omega^{11}&\omega^{ 3}&
\omega^{10}&\omega^{ 2}&\omega^{ 9}&\omega^{ 1}&\omega^{ 8}\\
\omega^{ 0}&\omega^{12}&\omega^{ 9}&\omega^{ 6}&\omega^{ 3}&
\omega^{ 0}&\omega^{12}&\omega^{ 9}&\omega^{ 6}&\omega^{ 3}&
\omega^{ 0}&\omega^{12}&\omega^{ 9}&\omega^{ 6}&\omega^{ 3}\\
\omega^{ 0}&\omega^{13}&\omega^{11}&\omega^{ 9}&\omega^{ 7}&
\omega^{ 5}&\omega^{ 3}&\omega^{ 1}&\omega^{14}&\omega^{12}&
\omega^{10}&\omega^{ 8}&\omega^{ 6}&\omega^{ 4}&\omega^{ 2}\\
\omega^{ 0}&\omega^{ 3}&\omega^{ 6}&\omega^{ 9}&\omega^{12}&
\omega^{ 0}&\omega^{ 3}&\omega^{ 6}&\omega^{ 9}&\omega^{12}&
\omega^{ 0}&\omega^{ 3}&\omega^{ 6}&\omega^{ 9}&\omega^{12}\\
\omega^{ 0}&\omega^{ 9}&\omega^{ 3}&\omega^{12}&\omega^{ 6}&
\omega^{ 0}&\omega^{ 9}&\omega^{ 3}&\omega^{12}&\omega^{ 6}&
\omega^{ 0}&\omega^{ 9}&\omega^{ 3}&\omega^{12}&\omega^{ 6}\\
\omega^{ 0}&\omega^{14}&\omega^{13}&\omega^{12}&\omega^{11}&
\omega^{10}&\omega^{ 9}&\omega^{ 8}&\omega^{ 7}&\omega^{ 6}&
\omega^{ 5}&\omega^{ 4}&\omega^{ 3}&\omega^{ 2}&\omega^{ 1}
\end{array}\right]
\end{equation*}
form an $\ETF(8,15)$, having $\abs{\ip{\bfphi_n}{\bfphi_{n'}}}=[\frac{G-D}{D(G-1)}]^{\frac12}=\frac14$ for all $n\neq n'$, and thus achieving equality in the Welch bound~\eqref{eq.Welch bound}.
Letting $S=4$ be the reciprocal of this bound,
we further have that $\calD$ is fine,
being disjoint from the unique subgroup of $\calG$ of order \smash{$H=\frac{G}{S+1}=3$}, namely $\calH=\set{0,5,10}$.
As such, any $8\times 5$ submatrix of $\bfPhi$ whose columns are indexed by a coset of $\calH^\perp$ have the property that these columns sum to zero \`{a} la~\eqref{eq.simplex sum} and moreover form a regular simplex for their $4$-dimensional span.
Here, under our identification of $\hat{\calG}$ with $\bbZ_{15}$,
\smash{$\calH^\perp$} corresponds to those $n\in\bbZ_{15}$ that have the property that $\omega^{nh}=1$ for all $h\in\calH=\set{0,5,10}$,
namely $\set{0,3,6,9,12}$.

Here, for any $n\in\bbZ_{15}$, the matrix $\bfPhi_n$ defined by~\eqref{eq.def of Phi_gamma} is the $8\times5$ submatrix of $\bfPhi$ whose columns are indexed by the $n$th coset of $\calH^\perp$, beginning with $n$, namely
\smash{$\bfPhi_n=\left[\begin{array}{ccccc}
\bfphi_{n}&\bfphi_{n+3}&\bfphi_{n+6}&\bfphi_{n+9}&\bfphi_{n+12}\end{array}\right]$}.
Each $\bfPhi_n$ is unique,
but $\bfPhi_n$ and $\bfPhi_{n'}$ are related via a $5\times 5$ permutation matrix whenever $n-n'\in\calH$.
Moreover, any choice of coset representatives of $\calG/\calH$ yields a partition of the ETF's vectors.
For example, concatenating $\bfPhi_0$, $\bfPhi_1$ and $\bfPhi_2$ yields the matrix obtained by perfectly shuffling the columns of $\bfPhi$:
\begin{equation}
\label{eq.8x15 Phi_gamma}
\left[\begin{array}{ccc}\bfPhi_0&\bfPhi_1&\bfPhi_2\end{array}\right]
=\frac1{\sqrt{8}}\left[\begin{array}{ccccc|ccccc|ccccc}
\omega^{ 0} & \omega^{ 3} & \omega^{ 6} & \omega^{ 9} & \omega^{12} & \omega^{ 6} & \omega^{ 9} & \omega^{12} & \omega^{ 0} & \omega^{ 3} & \omega^{12} & \omega^{ 0} & \omega^{ 3} & \omega^{ 6} & \omega^{ 9}\\
\omega^{ 0} & \omega^{ 3} & \omega^{ 6} & \omega^{ 9} & \omega^{12} & \omega^{11} & \omega^{14} & \omega^{ 2} & \omega^{ 5} & \omega^{ 8} & \omega^{ 7} & \omega^{10} & \omega^{13} & \omega^{ 1} & \omega^{ 4}\\
\omega^{ 0} & \omega^{ 6} & \omega^{12} & \omega^{ 3} & \omega^{ 9} & \omega^{ 7} & \omega^{13} & \omega^{ 4} & \omega^{10} & \omega^{ 1} & \omega^{14} & \omega^{ 5} & \omega^{11} & \omega^{ 2} & \omega^{ 8}\\
\omega^{ 0} & \omega^{ 6} & \omega^{12} & \omega^{ 3} & \omega^{ 9} & \omega^{12} & \omega^{ 3} & \omega^{ 9} & \omega^{ 0} & \omega^{ 6} & \omega^{ 9} & \omega^{ 0} & \omega^{ 6} & \omega^{12} & \omega^{ 3}\\
\omega^{ 0} & \omega^{ 9} & \omega^{ 3} & \omega^{12} & \omega^{ 6} & \omega^{13} & \omega^{ 7} & \omega^{ 1} & \omega^{10} & \omega^{ 4} & \omega^{11} & \omega^{ 5} & \omega^{14} & \omega^{ 8} & \omega^{ 2}\\
\omega^{ 0} & \omega^{ 9} & \omega^{ 3} & \omega^{12} & \omega^{ 6} & \omega^{ 3} & \omega^{12} & \omega^{ 6} & \omega^{ 0} & \omega^{ 9} & \omega^{ 6} & \omega^{ 0} & \omega^{ 9} & \omega^{ 3} & \omega^{12}\\
\omega^{ 0} & \omega^{12} & \omega^{ 9} & \omega^{ 6} & \omega^{ 3} & \omega^{ 9} & \omega^{ 6} & \omega^{ 3} & \omega^{ 0} & \omega^{12} & \omega^{ 3} & \omega^{ 0} & \omega^{12} & \omega^{ 9} & \omega^{ 6}\\
\omega^{ 0} & \omega^{12} & \omega^{ 9} & \omega^{ 6} & \omega^{ 3} & \omega^{14} & \omega^{11} & \omega^{ 8} & \omega^{ 5} & \omega^{ 2} & \omega^{13} & \omega^{10} & \omega^{ 7} & \omega^{ 4} & \omega^{ 1}
\end{array}\right].
\end{equation}
Meanwhile,
removing the first row of the character table of $\bbZ_5\cong\calG/\calH$ and then normalizing columns yields the synthesis operator~\eqref{eq.def of Psi} of a particularly  nice regular $4$-simplex:
\begin{equation}
\label{eq.8x15 Psi}
\bfPsi=\frac1{\sqrt{4}}\left[\begin{array}{ccccc}
\omega^{ 0}&\omega^{ 3}&\omega^{ 6}&\omega^{ 9}&\omega^{12}\\
\omega^{ 0}&\omega^{ 6}&\omega^{12}&\omega^{ 3}&\omega^{ 9}\\
\omega^{ 0}&\omega^{ 9}&\omega^{ 3}&\omega^{12}&\omega^{ 6}\\
\omega^{ 0}&\omega^{12}&\omega^{ 9}&\omega^{ 6}&\omega^{ 3}\\
\end{array}\right].
\end{equation}
By Theorem~\ref{thm.EITFF}(a),
each $\bfPhi_n$ can be decomposed as $\bfPhi_n=\bfE_n\bfPsi$ where $\bfE_n$ is the $8\times 4$ isometry defined in~\eqref{eq.def of E_gamma}.
In particular, the matrices $\bfPhi_0$, $\bfPhi_1$, $\bfPhi_2$ in~\eqref{eq.8x15 Phi_gamma} factor as products of
\begin{equation}
\label{eq.8x15 embeddings}
\bfE_0=\frac1{\sqrt{2}}\left[\begin{array}{cccc}
1 & 0 & 0 & 0\\
1 & 0 & 0 & 0\\
0 & 1 & 0 & 0\\
0 & 1 & 0 & 0\\
0 & 0 & 1 & 0\\
0 & 0 & 1 & 0\\
0 & 0 & 0 & 1\\
0 & 0 & 0 & 1
\end{array}\right],
\quad
\bfE_1=\frac1{\sqrt{2}}\left[\begin{array}{cccc}
\omega^{ 6} & 0 & 0 & 0\\
\omega^{11} & 0 & 0 & 0\\
0 & \omega^{ 7} & 0 & 0\\
0 & \omega^{12} & 0 & 0\\
0 & 0 & \omega^{13} & 0\\
0 & 0 & \omega^{ 3} & 0\\
0 & 0 & 0 & \omega^{ 9}\\
0 & 0 & 0 & \omega^{14}
\end{array}\right],
\quad
\bfE_2=\frac1{\sqrt{2}}\left[\begin{array}{cccc}
\omega^{12} & 0 & 0 & 0\\
\omega^{ 7} & 0 & 0 & 0\\
0 & \omega^{14} & 0 & 0\\
0 & \omega^{ 9} & 0 & 0\\
0 & 0 & \omega^{11} & 0\\
0 & 0 & \omega^{ 6} & 0\\
0 & 0 & 0 & \omega^{ 3}\\
0 & 0 & 0 & \omega^{13}
\end{array}\right],
\end{equation}
with $\bfPsi$, respectively;
here, the rows and columns of $\bfE_n$ are indexed by
$\calD=\set{6,11,7,12,13,3,9,14}$ and  $\calG/\calH\backslash\set{0}=\set{\overline{1},\overline{2},\overline{3},\overline{4}}$, respectively.
That is, each $\bfPhi_n$ is a regular simplex, being an isometric embedding of $\bfPsi$ into a particular $4$-dimensional subspace of $\bbC^\calD$,
namely into $\calU_{\overline{n}}=\rmC(\bfE_n)$.
Continuing, Theorem~\ref{thm.EITFF}(a) also implies every cross-Gram matrix $\bfE_n^*\bfE_{n'}^{}$ is diagonal.
Here, since $\omega^{5}+\omega^{10}=-1$,
\begin{equation}
\label{eq.8x15 cross-Grams}
\bfE_0^*\bfE_1^{}
=\bfE_1^*\bfE_2^{}
=-\frac12\left[\begin{array}{cccc}
\omega&0&0&0\\
0&\omega^2&0&0\\
0&0&\omega^8&0\\
0&0&0&\omega^4
\end{array}\right],
\quad
\bfE_0^*\bfE_2^{}
=-\frac12\left[\begin{array}{cccc}
\omega^2&0&0&0\\
0&\omega^4&0&0\\
0&0&\omega&0\\
0&0&0&\omega^8
\end{array}\right].
\end{equation}
This diagonality is crucial, since it means the singular values of $\bfE_n^*\bfE_{n'}^{}$---the cosines of the principal angles between $\calU_{\overline{n}}$ and $\calU_{\overline{n}'}$---are the absolute values of its diagonal entries.
In particular,
\eqref{eq.8x15 cross-Grams} implies that every principal angle $\theta$ between any pair of the subspaces $\calU_{\overline{0}}$, $\calU_{\overline{1}}$, $\calU_{\overline{2}}$ satisfies $\cos(\theta)=\frac12$,
meaning these three subspaces are equi-isoclinic.
In fact, Theorem~\ref{thm.EITFF}(b) gives that they form an EITFF for $\bbC^\calD$ since $\calD$ is an amalgam:
from Example~\ref{ex.8x15 fine},
recall that $\calD_g=\emptyset$ when $g\in\calH$ while for $g\notin\calH$,
$\calD_g$ is either $\set{5,10}$, $\set{0,5}$ or $\set{0,10}$, each of which is a difference set for $\calH=\set{0,5,10}$.

Since $\calD$ is an amalgam,
Theorem~\ref{thm.EITFF}(c) also applies:
for any $n\notin\calH^\perp$,
we construct the first column of the $5\times 5$ circulant conference matrix $\bfC_n$ by reshaping the diagonal entries of $\bfE_0^*\bfE_n^{}$ into a $4\times 1$ vector, padding it with a leading $0$ entry, and scaling the result by $\sqrt{S}=2$.
For example, in the $n=1$ and $n=2$ cases,
the cross-Gram matrices in~\eqref{eq.8x15 cross-Grams} yield the circulant conference matrices
\begin{equation*}
\bfC_1=-\left[\begin{array}{ccccc}
         0&\omega^{4}&\omega^{8}&\omega^{2}&\omega^{ }\\
\omega^{ }&         0&\omega^{4}&\omega^{8}&\omega^{2}\\
\omega^{2}&\omega^{ }&         0&\omega^{4}&\omega^{8}\\
\omega^{8}&\omega^{2}&\omega^{ }&         0&\omega^{4}\\
\omega^{4}&\omega^{8}&\omega^{2}&\omega^{ }&         0
\end{array}\right],\quad
\bfC_2=-\left[\begin{array}{ccccc}
         0&\omega^{8}&\omega^{ }&\omega^{4}&\omega^{2}\\
\omega^{2}&         0&\omega^{8}&\omega^{ }&\omega^{4}\\
\omega^{4}&\omega^{2}&         0&\omega^{8}&\omega^{ }\\
\omega^{ }&\omega^{4}&\omega^{2}&         0&\omega^{8}\\
\omega^{8}&\omega^{ }&\omega^{4}&\omega^{2}&         0
\end{array}\right].
\end{equation*}
Circulant conference matrices are interesting objects,
and there does not seem to be much literature regarding them.
Perhaps this is due to the long-known fact~\cite{StantonM76,Craigen94,TurekG16} that the only real-valued instances of such a matrix are $\pm[\begin{smallmatrix}0&1\\1&0\end{smallmatrix}]$.
The complex circulant conference matrices we obtain here are conceivably useful in certain applications where complex circulant Hadamard matrices are used,
like waveform design for radar.
While it is famously conjectured that real-valued $N\times N$ circulant Hadamard matrices only exist when $N\in\set{1,4}$---the \textit{circulant Hadamard conjecture}---infinite families of complex circulant Hadamard matrices are known;
such objects are equivalent to the \textit{constant-amplitude, zero-autocorrelation} (CAZAC) sequences of~\cite{BenedettoD07},
and arise from quadratic chirps as well as Bj\"{o}rck-Saffari sequences~\cite{BjorckS95}, for example.
\end{example}

Later on, we show that the above example is but the first member of an infinite family of known fine difference sets that are amalgams and so generate EITFFs and circulant conference matrices via Theorem~\ref{thm.EITFF}.
But first, we consider a subclass of amalgams that can be factored in terms of a certain type of relative difference set (RDS).
As we shall see, such RDSs lend themselves to a related and yet distinct method for constructing circulant conference matrices.

%%%%%%%%%%%%%%%%%%%%%%%%%%%%%%%%%%%%%%%%%%%%%%%%%%%%%%%%%%%%%%%%
\section{Composite difference sets and amalgams}
%%%%%%%%%%%%%%%%%%%%%%%%%%%%%%%%%%%%%%%%%%%%%%%%%%%%%%%%%%%%%%%%

\subsection{Composite difference sets}

In the previous section, we showed that the ECTFF~\eqref{eq.def of U_gamma} arising from a fine difference set $\calD$ is an EITFF for $\bbC^\calD$ if and only if $\calD$ is an amalgam, that is, if and only if for every $g\in\calG$, the set $\calD_g=\calH\cap(\calD-g)$ is a difference set for $\calH$.
In this section, we show that there are an infinite number of fine difference sets that are amalgams, as well as an infinite number that are not.
In fact, some but not all of these amalgams have an even stronger property,
namely that any two nontrivial $\calD_g$ are translations of each other.
For instance, from Example~\ref{ex.8x15 fine},
recall that $\calD=\set{6,11,7,12,13,3,9,14}$ is a difference set for $\bbZ_{15}$ that is $\calH$-fine where $\calH=\set{0,5,10}$ and, for $g\notin\calH$, $\calD_g$ is either $\set{5,10}$, $\set{0,5}$ or $\set{0,10}$.
In particular, there is a choice of representatives of the nonidentity cosets of $\calG/\calH$ such that the corresponding $\calD_g$ are equal:
$\calD_1=\calD_2=\calD_8=\calD_4=\set{5,10}$.
Applying this rationale in general,
we see that if all nontrivial $\calD_g$ are translates of some subset $\calB$ of $\calH$,
then there is a set $\calA$ of representatives of the nonidentity cosets of $\calG/\calH$ such that $\calD_a=\calB$ for all $a\in\calA$.
This in turn implies that $\calD$ can be written as
$\calD=\calA+\calB:=\set{a+b: a\in\calA,\ b\in\calB}$,
and in fact can be partitioned as $\calD=\sqcup_{a\in\calA}(a+\calB)$.
This means the characteristic function of $\calD$ factors as:
\begin{equation*}
\bfchi_\calD
=\sum_{a\in\calA}\chi_{a+\calB}
=\sum_{a\in\calA}\bfdelta_a*\bfchi_{\calB}
=\Bigparen{\sum_{a\in\calA}\bfdelta_a}*\bfchi_{\calB}
=\bfchi_\calA*\bfchi_\calB.
\end{equation*}
In the specific example above, $\calA=\set{1,2,8,4}$ and $\calB=\set{5,10}$.
We now give such sets a name:

\begin{definition}
\label{def.composite}
We say a difference set $\calD$ for a finite abelian group $\calG$ is \textit{composite} if it is $\calH$-fine for some subgroup $\calH$ of $\calG$---see Definition~\ref{def.fine}---and moreover there exist an $S$-element subset $\calA$ of $\calG$ and a difference set $\calB$ for $\calH$ such that $\bfchi_\calD=\bfchi_\calA*\bfchi_\calB$.
\end{definition}

Below, we show that the set $\calA$ here is necessarily an $\calH$-relative difference set (RDS) for $\calG$ with particularly simple parameters.
Here, recall that applying the quotient map $g\mapsto\overline{g}$ to any $\calH$-RDS produces a difference set for $\calG/\calH$.
As we shall see, quotienting the RDS $\calA$ arising from a composite difference set yields $\calG/\calH\backslash\set{\overline{0}}$.
For example, $\calA=\set{1,2,8,4}$ is a $\set{0,5,10}$-RDS for $\bbZ_{15}$ since its difference table is
\begin{equation*}
\begin{array}{c|cccc}
-& 1& 2& 8& 4\\\hline
1& 0&14& 8&12\\
2& 1& 0& 9&13\\
8& 7& 6& 0& 4\\
4& 3& 2&11& 0
\end{array},
\end{equation*}
and quotienting it by $\calH$ yields the nonzero members of $\calG/\calH\cong\bbZ_5$.
We further show that the cross-Gram matrices of the isometries~\eqref{eq.def of E_gamma} arising from a composite difference set have the remarkable property that their \textit{triple products} are scalar multiples of the identity.
For example, for the cross-Gram matrices~\eqref{eq.8x15 cross-Grams} arising from $\calD=\set{6,11,7,12,13,3,9,14}$, we have $(\bfE_0^*\bfE_1^{})(\bfE_1^*\bfE_2^{})(\bfE_2^*\bfE_0^{})$ is:
\begin{equation*}
-\frac18
\left[\begin{array}{cccc}
\omega&0&0&0\\
0&\omega^2&0&0\\
0&0&\omega^8&0\\
0&0&0&\omega^4
\end{array}\right]
\left[\begin{array}{cccc}
\omega&0&0&0\\
0&\omega^2&0&0\\
0&0&\omega^8&0\\
0&0&0&\omega^4
\end{array}\right]
\left[\begin{array}{cccc}
\omega^{13}&0&0&0\\
0&\omega^{11}&0&0\\
0&0&\omega^{14}&0\\
0&0&0&\omega^{7}
\end{array}\right]
=-\frac18
\left[\begin{array}{cccc}
1&0&0&0\\
0&1&0&0\\
0&0&1&0\\
0&0&0&1
\end{array}\right].
\end{equation*}
As we shall see, this implies that the subspaces~\eqref{eq.def of U_gamma} arising from a composite difference set are more than equi-isoclinic,
having orthogonal projection operators $\set{\bfP_n}_{n\in\calN}$ for which
$\bfP_{n_1}\bfP_{n_2}\bfP_{n_3}$ is always a scalar multiple of $\bfP_{n_1}\bfP_{n_3}$.

\begin{theorem}
\label{thm.composite}
Assume $\calD$ is a composite difference set for $\calG$,
and take $\calH$, $\calA$ and $\calB$ as in Definition~\ref{def.composite}.
For any $\gamma\in\hat{\calG}$,
define $\bfE_\gamma$ as in~\eqref{eq.def of E_gamma} and let
$\bfzeta_\gamma\in\bbC^{\calB}$, $\bfzeta_\gamma(b):=\frac{\sqrt{S}}{\sqrt{D}}\gamma(b)$.
Then:
\begin{enumerate}
\renewcommand{\labelenumi}{(\alph{enumi})}

\item
$\calB$ has cardinality $\frac{D}{S}$, and $\calA$ is an $\RDS(S+1,H,S,\frac{S-1}H)$ that is disjoint from $\calH$.

Moreover, $G-D$ divides $D-1$.

\item
$\calD$ is an amalgam---see Definition~\ref{def.amalgam}---with
$\calD_a=\calH\cap(\calD-a)=\calB$ for every $a\in\calA$.

\item
The orthogonal projection operators
$\set{\bfP_{\overline{\gamma}}}_{\overline{\gamma}\in\hat{\calG}/\calH^\perp}$
onto the subspaces~\eqref{eq.def of U_gamma} satisfy
\begin{equation}
\label{eq.composite projections}
\smash{\ip{\bfzeta_{\gamma_1}}{\bfzeta_{\gamma_3}}
\bfP_{\overline{\gamma}_1}\bfP_{\overline{\gamma}_2}\bfP_{\overline{\gamma}_3}
=
\ip{\bfzeta_{\gamma_1}}{\bfzeta_{\gamma_2}}
\ip{\bfzeta_{\gamma_2}}{\bfzeta_{\gamma_3}}
\bfP_{\overline{\gamma}_1}\bfP_{\overline{\gamma}_3},
\quad
\forall\ \gamma_1,\gamma_2,\gamma_3\in\hat{\calG},}
\end{equation}
where
$
\abs{\ip{\bfzeta_{\gamma}}{\bfzeta_{\gamma'}}}
=\left\{\begin{array}{cl}
1,&               \ \overline{\gamma}=\overline{\gamma}',\\
\frac1{\sqrt{S}},&\ \overline{\gamma}\neq\overline{\gamma}'.
\end{array}\right.
$
\end{enumerate}
\end{theorem}

\begin{proof}
For (a), recall from Section~2 that a $D$-element subset $\calD$ of $\calG$ is an $\calH$-RDS for $\calG$ if and only if it satisfies~\eqref{eq.relative Fourier};
here, since $\calA$ has cardinality $S$ and $G=H(S+1)$,
$\Lambda$ is necessarily \smash{$\tfrac{S(S-1)}{G-H}=\frac{S-1}{H}$}
and $\calA$ is an $\calH$-RDS for $\calG$ if and only if
\begin{equation}
\label{eq.pf of composite 1}
\abs{(\bfF^*\bfchi_\calA)(\gamma)}^2
=\left\{\begin{array}{cl}
1,&\ \gamma\in\calH^\perp,\ \gamma\neq1,\\
S,&\ \gamma\notin\calH^\perp.
\end{array}\right.
\end{equation}
To show this holds,
note that since
$\bfchi_\calD=\bfchi_\calA*\bfchi_\calB$ we have
$(\bfF^*\bfchi_\calD)(\gamma)=(\bfF^*\bfchi_\calA)(\gamma)(\bfF^*\bfchi_\calB)(\gamma)$ for all $\gamma\in\hat{\calG}$.
Letting $\gamma=1$ gives $D=S\#(\calB)$ and so $\calB$ is a difference set for $\calH$ of cardinality $\frac DS$.
Lemma~\ref{lem.small difference sets} then gives that
\smash{$\abs{(\bfF^*\bfchi_{\calB})(\gamma)}^2=\tfrac{D^2}{S^3}$} for any $\gamma\notin\calH^\perp$.
When combined with the fact from~\eqref{eq.diff set Fourier} that
$\abs{(\bfF^*\bfchi_\calD)(\gamma)}=\frac{D}{S}$ for all $\gamma\neq1$,
this implies
\begin{equation*}
\abs{(\bfF^*\bfchi_\calA)(\gamma)}^2
=\tfrac{\abs{(\bfF^*\bfchi_\calD)(\gamma)}^2}{\abs{(\bfF^*\bfchi_\calB)(\gamma)}^2}
=\tfrac{D^2}{S^2}\tfrac{S^3}{D^2}
=S,
\quad
\forall\ \gamma\notin\calH^\perp.
\end{equation*}
Meanwhile, the fact that $\calB$ is a $\frac{D}{S}$-element subset of $\calH$ implies
$(\bfF^*\bfchi_\calB)(\gamma)=\frac{D}{S}$ for any $\gamma\in\calH^\perp$.
As such, for any $\gamma\in\calH^\perp$, $\gamma\neq1$,
combining this with Theorem~\ref{thm.fine}(ii) gives
\begin{equation*}
(\bfF^*\bfchi_\calA)(\gamma)
=\tfrac{(\bfF^*\bfchi_\calD)(\gamma)}{(\bfF^*\bfchi_\calB)(\gamma)}
=-\tfrac{D}{S}\tfrac{S}{D}
=-1.
\end{equation*}
Thus, \eqref{eq.pf of composite 1} indeed holds, meaning $\calA$ is an $S$-element $\calH$-RDS for the group $\calG$ of order $G=H(S+1)$.
In particular, $\calA$ is an $\RDS(S+1,H,S,\frac{S-1}{H})$.
Moreover, $\calA$ is disjoint from $\calH$ since
\begin{equation*}
\tfrac GH\ip{\bfchi_H}{\bfchi_\calA}
=\tfrac1H\ip{\bfF^*\bfchi_\calH}{\bfF^*\bfchi_\calA}
=\ip{\bfchi_{\calH^\perp}}{\bfF^*\bfchi_\calA}
=S+\sum_{\substack{\gamma\in\calH^\perp\\\gamma\neq1}}(\bfF^*\bfchi_\calA)(\gamma)
=S+S(-1)
=0.
\end{equation*}
Here, for any $g\notin\calH$,
the set $\set{(a,a')\in\calA\times\calA: g=a-a'}$ has cardinality \smash{$\frac{S-1}{H}=\frac{S^2-1}{G}=\frac{D-1}{G-D}$},
and so $G-D$ divides $D-1$.

For (b),
note that since $\calA$ is an $\calH$-$\RDS(S+1,H,S,\frac{S-1}H)$ that is disjoint from $\calH$, quotienting it by $\calH$ yields an $S$-element difference set $\overline{\calA}=\set{\overline{a}: a\in\calA}$ for the group $\calG/\calH$ of order $S+1$ that does not contain $\overline{0}$.
Thus, $\overline{\calA}=\calG/\calH\backslash\set{\overline{0}}$,
and so
$\calA$ is a set of representatives of the nonidentity cosets of $\calH$.
Moreover, since
$\bfchi_\calD
=\bfchi_\calA*\bfchi_\calB
=\sum_{a\in\calA}\bfdelta_a*\bfchi_\calB
=\sum_{a\in\calA}\bfchi_{a+\calB}$,
the set $\calD$ can be partitioned as $\sqcup_{a\in\calA}(a+\calB)$
where $\calB$ is a subset of $\calH$.
This implies that for any $g\notin\calH$,
taking the unique $a\in\calA$ such that $\overline{a}=\overline{g}$, we have
\begin{equation*}
\calD_g
=\calH\cap(\calD-g)
=\calH\cap\Bigparen{\,\bigsqcup_{a'\in\calA}(a'+\calB)-g}
=\bigsqcup_{a'\in\calA}\bigset{\calH\cap[(a'-g)+\calB]}
=(a-g)+\calB
\end{equation*}
is a difference set for $\calH$, being a shift of the difference set $\calB$.
Since $\calD_g$ is empty whenever $g\in\calH$, we thus have that every $\calD_g$ is a difference set for $\calH$, namely that $\calD$ is an amalgam.
Moreover, in the special case where $g=a$, the above equation becomes $\calD_a=\calB$.

For (c), note
\smash{$\set{\bfzeta_\gamma}_{\gamma\in\hat{\calG}}$} consists of $S+1$ copies of the harmonic ETF that arises from $\calB$ being a difference set for $\calH$.
Indeed,
\smash{$\ip{\bfzeta_{\gamma}}{\bfzeta_{\gamma'}}
=\tfrac{S}{D}\sum_{b\in\calB}(\gamma^{-1}\gamma')(b)
=\tfrac{S}{D}(\bfF^*\bfchi_\calB)(\gamma(\gamma')^{-1})$}
for any $\gamma,\gamma'\in\hat{\calG}$,
and so Lemma~\ref{lem.small difference sets} gives
\begin{equation*}
\abs{\ip{\bfzeta_{\gamma}}{\bfzeta_{\gamma'}}}^2
=\tfrac{S^2}{D^2}\abs{(\bfF^*\bfchi_\calB)(\gamma(\gamma')^{-1})}^2
=\left\{\begin{array}{cl}
1,      &\ \gamma(\gamma')^{-1}\in\calH^\perp,\\
\frac1S,&\ \gamma(\gamma')^{-1}\notin\calH^\perp.
\end{array}\right.
\end{equation*}
Moreover, for any $a\in\calA$,
we have $\set{d\in\calD: \overline{d}=\overline{a}}=a+\calD_a=a+\calB=\set{a+b: b\in\calB}$.
Combining these facts with Theorem~\ref{thm.EITFF}(a),
we find that for any $\gamma,\gamma'\in\hat{\calG}$,
$\bfE_\gamma^*\bfE_{\gamma'}^{}$ is a diagonal matrix where
\begin{equation*}
(\bfE_\gamma^*\bfE_{\gamma'}^{})(\overline{a},\overline{a})
=\tfrac{S}{D}\sum_{b\in\calB}(\gamma^{-1}\gamma')(a+b)
=\tfrac{S}{D}(\gamma^{-1}\gamma')(a)(\bfF^*\bfchi_\calB)(\gamma(\gamma')^{-1})
=(\gamma^{-1}\gamma')(a)\ip{\bfzeta_{\gamma}}{\bfzeta_{\gamma'}}
\end{equation*}
for any $a\in\calA$.
As such, for any $\gamma_1,\gamma_2,\gamma_3\in\hat{\calG}$,
$\bfE_{\gamma_3}^*\bfE_{\gamma_1}^{}\bfE_{\gamma_1}^*\bfE_{\gamma_3}^{}$ and
$\bfE_{\gamma_1}^*\bfE_{\gamma_2}^{}\bfE_{\gamma_2}^*\bfE_{\gamma_3}^{}\bfE_{\gamma_3}^*\bfE_{\gamma_1}^{}$ are diagonal matrices whose $\overline{a}$th diagonal entries are
\begin{align*}
(\bfE_{\gamma_3}^*\bfE_{\gamma_1}^{}\bfE_{\gamma_1}^*\bfE_{\gamma_3}^{})(\overline{a},\overline{a})
&=
(\gamma_3^{-1}\gamma_1^{})(a)
(\gamma_1^{-1}\gamma_3^{})(a)
\ip{\bfzeta_{\gamma_3}}{\bfzeta_{\gamma_1}}
\ip{\bfzeta_{\gamma_1}}{\bfzeta_{\gamma_3}},\\
(\bfE_{\gamma_1}^*\bfE_{\gamma_2}^{}\bfE_{\gamma_2}^*\bfE_{\gamma_3}^{}\bfE_{\gamma_3}^*\bfE_{\gamma_1}^{})(\overline{a},\overline{a})
&=
(\gamma_1^{-1}\gamma_2^{})(a)
(\gamma_2^{-1}\gamma_3^{})(a)
(\gamma_3^{-1}\gamma_1^{})(a)
\ip{\bfzeta_{\gamma_1}}{\bfzeta_{\gamma_2}}
\ip{\bfzeta_{\gamma_2}}{\bfzeta_{\gamma_3}}
\ip{\bfzeta_{\gamma_3}}{\bfzeta_{\gamma_1}},
\end{align*}
respectively.
Here, the $a$-dependent terms perfectly cancel, yielding
\begin{equation}
\label{eq.pf of composite 3}
\bfE_{\gamma_3}^*\bfE_{\gamma_1}^{}\bfE_{\gamma_1}^*\bfE_{\gamma_3}^{}
=\abs{\ip{\bfzeta_{\gamma_1}}{\bfzeta_{\gamma_3}}}^2\bfI,
\quad
\bfE_{\gamma_1}^*\bfE_{\gamma_2}^{}
\bfE_{\gamma_2}^*\bfE_{\gamma_3}^{}
\bfE_{\gamma_3}^*\bfE_{\gamma_1}^{}
=
\ip{\bfzeta_{\gamma_1}}{\bfzeta_{\gamma_2}}
\ip{\bfzeta_{\gamma_2}}{\bfzeta_{\gamma_3}}
\ip{\bfzeta_{\gamma_3}}{\bfzeta_{\gamma_1}}
\bfI.
\end{equation}
As such, right-multiplying the second equation by
$\frac1{\ip{\bfzeta_{\gamma_3}}{\bfzeta_{\gamma_1}}}\bfE_{\gamma_1}^*\bfE_{\gamma_3}^{}$ gives
\begin{equation*}
\ip{\bfzeta_{\gamma_1}}{\bfzeta_{\gamma_3}}
\bfE_{\gamma_1}^*\bfE_{\gamma_2}^{}
\bfE_{\gamma_2}^*\bfE_{\gamma_3}^{}
=
\ip{\bfzeta_{\gamma_1}}{\bfzeta_{\gamma_2}}
\ip{\bfzeta_{\gamma_2}}{\bfzeta_{\gamma_3}}
\bfE_{\gamma_1}^*\bfE_{\gamma_3}^{},
\end{equation*}
as claimed.
(The interested reader can verify that this single property actually implies both statements of~\eqref{eq.pf of composite 3} as special cases.)
Moreover, since every $\bfE_\gamma$ satisfies $\bfE_\gamma^*\bfE_\gamma^{}=\bfI$,
this equates to having
$
\ip{\bfzeta_{\gamma_1}}{\bfzeta_{\gamma_3}}
\bfE_{\gamma_1}^{}
\bfE_{\gamma_1}^*\bfE_{\gamma_2}^{}
\bfE_{\gamma_2}^*\bfE_{\gamma_3}^{}
\bfE_{\gamma_3}^*
=
\ip{\bfzeta_{\gamma_1}}{\bfzeta_{\gamma_2}}
\ip{\bfzeta_{\gamma_2}}{\bfzeta_{\gamma_3}}
\bfE_{\gamma_1}^{}
\bfE_{\gamma_1}^*\bfE_{\gamma_3}^{}
\bfE_{\gamma_3}^*
$,
namely to having~\eqref{eq.composite projections}.
In the special case where $\overline{\gamma}_1=\overline{\gamma}_3\neq\overline{\gamma}_2$,
this reduces to the fact that~\eqref{eq.def of U_gamma} is an EITFF for $\bbC^\calD$, as previously observed in Theorem~\ref{thm.EITFF}(b).
\end{proof}

There is a precedent for orthogonal projection operators that satisfy~\eqref{eq.composite projections}.
To elaborate, from Section~2, recall that if $\set{\bfdelta_m}_{m\in\calM}$ is an orthonormal basis for $\bbH$ and for each $m\in\calM$,
\smash{$\set{\bfphi_n^{(m)}}_{n\in\calN}$} is any $\ETF(D,N)$ for $\bbK$,
then $\set{\calU_n}_{n\in\calN}$, \smash{$\calU_n:=\Span\set{\bfdelta_m\otimes\bfphi_n^{(m)}}_{m\in\calM}$}
is an EITFF for $\bbK\otimes\bbH$.
In the special case where these $M$ ETFs are identical,
we have that for any $n\in\calN$, the orthogonal projection operator onto $\calU_n$ is
\begin{equation*}
\bfP_n
=\sum_{m\in\calM}(\bfdelta_m\otimes\bfphi_n)(\bfdelta_m\otimes\bfphi_n)^*
=\Bigparen{\sum_{m\in\calM}\bfdelta_m^{}\bfdelta_m^*}\otimes\bfphi_n^{}\bfphi_n^*
=\bfI\otimes\bfphi_n^{}\bfphi_n^*.
\end{equation*}
In particular, for any $n_1,n_2,n_3\in\calN$,
\begin{align*}
\bfP_{n_1}\bfP_{n_3}
&=\bfI\otimes(\bfphi_{n_1}^{}\bfphi_{n_1}^*\bfphi_{n_3}^{}\bfphi_{n_3}^*)
=\ip{\bfphi_{n_1}}{\bfphi_{n_3}}(\bfI\otimes\bfphi_{n_1}^{}\bfphi_{n_3}^{*}),\\
\bfP_{n_1}\bfP_{n_2}\bfP_{n_3}
&=\bfI\otimes(\bfphi_{n_1}^{}\bfphi_{n_1}^*\bfphi_{n_2}^{}\bfphi_{n_2}^*\bfphi_{n_3}^{}\bfphi_{n_3}^*)
=\ip{\bfphi_{n_1}}{\bfphi_{n_2}}\ip{\bfphi_{n_2}}{\bfphi_{n_3}}(\bfI\otimes\bfphi_{n_1}^{}\bfphi_{n_3}^{*}),
\end{align*}
and so
$
\ip{\bfphi_{n_1}}{\bfphi_{n_3}}
\bfP_{n_1}\bfP_{n_2}\bfP_{n_3}
=
\ip{\bfphi_{n_1}}{\bfphi_{n_2}}
\ip{\bfphi_{n_2}}{\bfphi_{n_3}}
\bfP_{n_1}\bfP_{n_3}
$.
The similarity between this and~\eqref{eq.composite projections} is not a coincidence.
To explain, note that for any fine difference set,
concatenating the matrices \smash{$\set{\bfE_\gamma}_{\gamma\in\hat{\calG}/\calH^\perp}$} from Theorem~\ref{thm.EITFF}
over any choice of representatives of the cosets of \smash{$\calH^\perp$} and then perfectly shuffling columns yields a $D\times HS$ block-diagonal matrix,
specifically an $S\times S$ array of blocks of size $\frac DS\times H$.
For example, applying this process to~\eqref{eq.8x15 embeddings} yields an $8\times 12$ block diagonal matrix, namely a $4\times 4$ array whose four $2\times 3$ diagonal blocks are
\begin{equation*}
\left[\begin{array}{ccc}
1 & \omega^{ 6} & \omega^{12} \\
1 & \omega^{11} & \omega^{ 7}
\end{array}\right],
\quad
\left[\begin{array}{ccc}
1 & \omega^{ 7} & \omega^{14} \\
1 & \omega^{12} & \omega^{ 9}
\end{array}\right],
\quad
\left[\begin{array}{ccc}
1 & \omega^{13} & \omega^{11} \\
1 & \omega^{ 3} & \omega^{ 6}
\end{array}\right],
\quad
\left[\begin{array}{ccc}
1 & \omega^{ 9} & \omega^{ 3} \\
1 & \omega^{14} & \omega^{13}
\end{array}\right].
\end{equation*}
Disregarding scalar multiples of columns,
this is simply four copies of the harmonic ETF that arises from the difference set
$\calB=\set{5,10}$ for $\calH=\set{0,5,10}$.
To see the degree to which this behavior holds in general,
note that for any choice of representatives of the nonidentity cosets of $\calG/\calH$,
dividing every column of $\bfE_\gamma$ by the corresponding value of $\gamma(g)$
yields the matrix whose $(d,\overline{g})$th entry is
\begin{equation*}
\gamma^{-1}(g)\bfE_\gamma(d,\overline{g})
=\gamma^{-1}(g)\tfrac{\sqrt{S}}{\sqrt{D}}\left\{\begin{array}{cl}
\gamma(d),&\ \overline{d}=\overline{g},\\
0,        &\ \overline{d}\neq\overline{g}.
\end{array}\right\}
=\tfrac{\sqrt{S}}{\sqrt{D}}\left\{\begin{array}{cl}
\gamma(d-g),&\ d-g\in\calD_g,\\
0,          &\ d-g\notin\calD_g.
\end{array}\right.
\end{equation*}
Moreover, when $d-g\in\calD_g$, the fact that $\calD_g$ is a subset of $\calH$ implies that the value of $\gamma(d-g)$ only depends on the coset of $\calH^\perp$ to which $\gamma$ belongs.
At the same time, $\hat{\calG}/\calH^\perp$ is isomorphic to the dual of $\calH$ via the mapping that identifies $\overline{\gamma}$ with the restriction of $\gamma$ to $\calH$.
And, under this identification,
the synthesis operator of the harmonic tight frame
$\set{\bfphi_{\overline{\gamma}}^{(g)}}_{\overline{\gamma}\in\hat{\calG}/\calH^\perp}$
that arises from regarding $\calD_g$ as a subset of $\calH$ is
\begin{equation*}
\bfPhi^{(g)}\in\bbC^{\calD_g\times\hat{\calG}/\calH^\perp},
\quad
\bfPhi^{(g)}(h,\overline{\gamma})
:=\tfrac{\sqrt{S}}{\sqrt{D}}\gamma(h).
\end{equation*}
Comparing the previous two equations gives
\begin{equation*}
\gamma^{-1}(g)\bfE_\gamma(d,\overline{g})
=\left\{\begin{array}{cl}
\bfPhi^{(g)}(d-g,\overline{\gamma}),&\ d-g\in\calD_g,\\
0,                                           &\ d-g\notin\calD_g.
\end{array}\right.
\end{equation*}
As such,
the union of the columns of $\bfE_\gamma$ over any choice of representatives of the cosets of $\hat{\calG}/\calH^\perp$ is equivalent---via permutations and unimodular scalings---to a union of \smash{$\set{\bfdelta_{\overline{g}}\otimes\bfphi_{\overline{\gamma}}^{(g)}}_{\overline{\gamma}\in\hat{\calG}/\calH^\perp}$}
over any choice of representatives of the nonidentity cosets of $\calG/\calH$;
here, $\bfdelta_{\overline{g}}$ is the $\overline{g}$th standard basis in the $S$-dimensional space $\bbC^{\calG/\calH\backslash\set{\overline{0}}}$.
Theorem~\ref{thm.EITFF} gives that~\eqref{eq.def of U_gamma} is an EITFF for $\bbC^\calD$ if and only if each $\calD_g$ is a difference set for $\calH$,
namely if and only if every \smash{$\set{\bfphi_{\overline{\gamma}}^{(g)}}_{\overline{\gamma}\in\hat{\calG}/\calH^\perp}$} is an ETF for $\bbC^{\calD_g}$.
Altogether, we see that every harmonic EITFF produced by Theorem~\ref{thm.EITFF} is but a disguised version of the classical EITFF construction given in~\eqref{eq.ETF tensor ONB}.

In the special case that $\calD$ is composite,
then taking our set of representatives of $\calG/\calH\backslash\set{\overline{0}}$ to be $\calA$,
the fact that $\calD_a=\calB$ for all $a\in\calA$ implies all \smash{$\set{\bfphi_{\overline{\gamma}}^{(a)}}_{\overline{\gamma}\in\hat{\calG}/\calH^\perp}$} are equal, being the harmonic ETF arising from the difference set $\calB$ of $\calH$.
In this case, the union of the columns of $\bfE_\gamma$ over any choice of representatives of the cosets of $\hat{\calG}/\calH^\perp$ is thus equivalent to a union of all tensor products of the standard basis of $\bbC^\calA$ with this harmonic ETF for $\bbC^\calB$.
From this perspective,
it is thus not surprising that an EITFF that arises from a composite difference set obeys~\eqref{eq.composite projections}.

For example,
since $\calD=\set{6,11,7,12,13,3,9,14}$ is composite with $\calA=\set{1,2,8,4}$, dividing the four columns of each $\bfE_n$ in~\eqref{eq.8x15 embeddings} by
$\set{\omega^n,\omega^{2n},\omega^{8n},\omega^{4n}}$, respectively,
and then concatenating and shuffling the resulting matrices yields a direct sum of four copies of the $2\times 3$ synthesis operator of the harmonic ETF that arises from $\calB=\set{5,10}$ being a difference set for $\calH=\set{0,5,10}$,
namely
\begin{equation*}
\bfPhi^{(1)}
=\bfPhi^{(2)}
=\bfPhi^{(8)}
=\bfPhi^{(3)}
=\frac1{\sqrt{2}}\left[\begin{array}{ccc}
1&\omega^{ 5}&\omega^{10}\\
1&\omega^{10}&\omega^{ 5}
\end{array}\right].
\end{equation*}
Here, it is simply a coincidence that the harmonic ETF arising from $\calB$ happens to itself be a regular simplex.
Later on, we provide an explicit construction of a composite difference set where $\calB$ turns out to be the complement of an arbitrary Singer difference set.

\subsection{Simplicial relative difference sets}

From Theorem~\ref{thm.composite}, recall that any composite difference set $\calD$ for $\calG$ yields a relative difference set $\calA$, specifically an $\RDS(S+1,H,S,\frac{S-1}{H})$ that is disjoint from $\calH$.
As we now explain, any relative difference set of this type has some remarkable properties, regardless of whether it arises from a composite difference set in this way.
Here, quotienting such an RDS by $\calH$ produces the difference set that consists of all $S$ nonidentity members of the group $\calG/\calH$ of order $S+1$.
That is, any such $\calA$ is a set of representatives of the nonidentity cosets of $\calH$.
We give such sets a name:

\begin{definition}
\label{def.simplicial}
Let $\calH$ be a subgroup of a finite abelian group $\calG$.
A subset $\calA$ of $\calG$ is a \textit{simplicial $\calH$-RDS} if it is an $\calH$-RDS for $\calG$ that is also a set of representatives of the nonidentity cosets of $\calH$.
Equivalently, $\calA$ is an $\calH$-$\RDS(S+1,H,S,\frac{S-1}{H})$ that is disjoint from $\calH$.
\end{definition}

We remark that if $\calA$ is a simplicial RDS for $\calG$,
then if $\calK$ is any subgroup of the corresponding subgroup $\calH$ of $\calG$,
then quotienting by $\calK$ transforms the simplicial $\RDS(S+1,H,S,\frac{S-1}{H})$ into a simplicial $\RDS(S+1,\frac{H}{K},S,\frac{K}{H}(S-1))$.
Below we show that the harmonic tight frame arising from a simplicial RDS is comprised of regular simplices, and moreover,
that these simplices are \textit{mutually unbiased} in the quantum-information-theoretic sense.
For example, extracting the rows of the character table of $\calG=\bbZ_{15}$ indexed by members of the simplicial RDS $\calA=\set{1,2,8,4}$, and grouping the resulting normalized columns according to cosets of $\calH^\perp$ gives
\begin{equation*}
\left[\begin{array}{ccc}\bfXi_0&\bfXi_1&\bfXi_2\end{array}\right]
=\frac1{\sqrt{4}}\left[\begin{array}{ccccc|ccccc|ccccc}
\omega^{ 0}&\omega^{ 3}&\omega^{ 6}&\omega^{ 9}&\omega^{12}&
\omega^{ 1}&\omega^{ 4}&\omega^{ 7}&\omega^{10}&\omega^{13}&
\omega^{ 2}&\omega^{ 5}&\omega^{ 8}&\omega^{11}&\omega^{14}\\
\omega^{ 0}&\omega^{ 6}&\omega^{12}&\omega^{ 3}&\omega^{ 9}&
\omega^{ 2}&\omega^{ 8}&\omega^{14}&\omega^{ 5}&\omega^{11}&
\omega^{ 4}&\omega^{10}&\omega^{ 1}&\omega^{ 7}&\omega^{13}\\
\omega^{ 0}&\omega^{ 9}&\omega^{ 3}&\omega^{12}&\omega^{ 6}&
\omega^{ 8}&\omega^{ 2}&\omega^{11}&\omega^{ 5}&\omega^{14}&
\omega^{ 1}&\omega^{10}&\omega^{ 4}&\omega^{13}&\omega^{ 7}\\
\omega^{ 0}&\omega^{12}&\omega^{ 9}&\omega^{ 6}&\omega^{ 3}&
\omega^{ 4}&\omega^{ 1}&\omega^{13}&\omega^{10}&\omega^{ 7}&
\omega^{ 8}&\omega^{ 5}&\omega^{ 2}&\omega^{14}&\omega^{11}
\end{array}\right],
\end{equation*}
namely three modulated versions of the regular $4$-simplex whose synthesis operator $\bfPsi$ is given in~\eqref{eq.8x15 Psi}.
Here, any columns from distinct simplices have an inner product of modulus $\frac12$.

Below, we further prove that every simplicial RDS for $\calG$ yields a complex $\calG/\calH$-circulant conference matrix.
In the special case where $\calA$ is a simplicial RDS arising from a composite difference set, this construction reduces to a unimodular scalar multiple of the construction given in Theorem~\ref{thm.EITFF}(c).
However, as later examples will demonstrate,
both this construction here and that of Theorem~\ref{thm.EITFF} are nontrivial generalizations of this common case,
and each is capable of producing instances of circulant conference matrices that the other is not.
Here, beginning with any simplicial RDS, such as $\calA=\set{1,2,8,4}$ for $\calG=\bbZ_{15}$ for example,
we modulate its characteristic function by a character of $\calG$ that does not lie in $\calH^\perp$, e.g.\ $(0,\omega,\omega^2,0,\omega^4,0,0,0,\omega^8,0,0,0,0,0,0)$ where
\smash{$\omega=\rme^{\frac{2\pi\rmi}{15}}$}.
We then periodize this vector into one that is indexed by $\calG/\calH$, e.g.\ $(0,\omega,\omega^2,\omega^8,\omega)$.
As we prove below, this new vector is orthogonal to each of its translates.
As such, the $(\calG/\calH)$-circulant matrix with this vector as its first column,
e.g.,\
\begin{equation*}
\left[\begin{array}{ccccc}
         0&\omega^{4}&\omega^{8}&\omega^{2}&\omega^{ }\\
\omega^{ }&         0&\omega^{4}&\omega^{8}&\omega^{2}\\
\omega^{2}&\omega^{ }&         0&\omega^{4}&\omega^{8}\\
\omega^{8}&\omega^{2}&\omega^{ }&         0&\omega^{4}\\
\omega^{4}&\omega^{8}&\omega^{2}&\omega^{ }&         0
\end{array}\right],
\end{equation*}
is a conference matrix.
Since this construction is valid for any character that does not lie in $\calH^\perp$,
we may, for example, raise each entry of the above matrix to any power not divisible by $3$ to obtain another such matrix.

\begin{theorem}
\label{thm.simplicial}
Let $\calH$ be a subgroup of a finite abelian group $\calG$.
Let $\calA$ be a subset of $\calG$ with $\#(\calA)=S=\frac GH-1$.
Let $\set{\bfxi_\gamma}_{\gamma\in\hat{\calG}}\subseteq\bbC^\calA$, $\bfxi_\gamma(a):=\frac1{\sqrt{S}}\gamma(a)$.
Then, the following are equivalent:
\begin{enumerate}
\renewcommand{\labelenumi}{(\roman{enumi})}
\item
$\calA$ is a simplicial $\calH$-RDS for $\calG$; see Definition~\ref{def.simplicial}.
\item
$(\bfF^*\bfchi_\calA)(\gamma)=-1$ for all $\gamma\in\calH^\perp$, $\gamma\neq1$ while
$\abs{(\bfF^*\bfchi_\calA)(\gamma)}=\sqrt{S}$ for all $\gamma\notin\calH^\perp$.
\item
For each $\gamma\in\hat{\calG}$,
$\set{\bfxi_{\gamma'}}_{\gamma'\in\gamma\calH^\perp}$ is a regular $S$-simplex whose vectors sum to zero, and moreover these simplices are mutually unbiased in the sense that
$\abs{\ip{\bfxi_\gamma}{\bfxi_{\gamma'}}}$ is constant over all
$\overline{\gamma}\neq\overline{\gamma}'$.
\end{enumerate}
Moreover, in this case,
$\abs{\ip{\bfxi_\gamma}{\bfxi_{\gamma'}}}=\frac1{\sqrt{S}}$
whenever $\overline{\gamma}\neq\overline{\gamma}'$, and for any $\gamma\notin\calH^\perp$,
\begin{equation*}
\hat{\bfC}_\gamma\in\bbC^{\calG/\calH\times\calG/\calH},
\quad
\hat{\bfC}_\gamma(\overline{g},\overline{g}')
:=\sum_{\substack{a\in\calA\\\overline{a}=\overline{g}-\overline{g}'}}\gamma(a),
\end{equation*}
is a circulant conference matrix.
In the special case where $\calA$ arises from a composite difference set $\calD$ via Theorem~\ref{thm.composite},
$\hat{\bfC}$ is a unimodular scalar multiple of the matrix $\bfC$ constructed in Theorem~\ref{thm.EITFF}(c).
\end{theorem}

\begin{proof}
Since $S=\frac{G}{H}-1$,
we have $G=H(S+1)$ and moreover \smash{$\frac{S(S-1)}{G-H}=\frac{S-1}{H}$}.
Thus, $\calA$ is an $\calH$-RDS for $\calG$ if and only if it is an $\calH$-$\RDS(S+1,H,S,\frac{S-1}{H})$.
Moreover, in this case, \eqref{eq.relative Fourier} and~\eqref{eq.relative harmonic} thus give that $\calA$ is an $\calH$-RDS if and only if
\begin{equation}
\label{eq.pf of simplicial 1}
\abs{(\bfF^*\bfchi_\calA)(\gamma)}^2
=\left\{\begin{array}{cl}
1,&\ \gamma\in\calH^\perp,\ \gamma\neq1,\\
S,&\ \gamma\notin\calH^\perp,
\end{array}\right.
\quad
\text{i.e.,}
\quad
\abs{\ip{\bfxi_\gamma}{\bfxi_{\gamma'}}}
=\left\{\begin{array}{cl}
\frac1S,         &\ \overline{\gamma}=\overline{\gamma}',\ \gamma\neq\gamma',\smallskip\\
\frac1{\sqrt{S}},&\ \overline{\gamma}\neq\overline{\gamma}'.
\end{array}\right.
\end{equation}
Moreover, $\calA$ is disjoint from $\calH$ if and only if
\begin{equation}
\label{eq.pf of simplicial 2}
0
=\tfrac GH\ip{\bfchi_H}{\bfchi_\calA}
=\tfrac1H\ip{\bfF^*\bfchi_\calH}{\bfF^*\bfchi_\calA}
=\ip{\bfchi_{\calH^\perp}}{\bfF^*\bfchi_\calA}
=S+\sum_{\substack{\gamma\in\calH^\perp\\\gamma\neq1}}(\bfF^*\bfchi_\calA)(\gamma).
\end{equation}
In particular, $\calA$ is a simplicial $\calH$-RDS if and only if it is an $\calH$-$\RDS(S+1,H,S,\frac{S-1}{H})$ that is disjoint from $\calH$,
namely if and only if it satisfies both~\eqref{eq.pf of simplicial 1} and~\eqref{eq.pf of simplicial 2}.

(i $\Leftrightarrow$ ii)
In light of the above facts,
it suffices to prove that (ii) holds if and only if $\calA$ satisfies both~\eqref{eq.pf of simplicial 1} and~\eqref{eq.pf of simplicial 2}.
Here, (ii) immediately implies both~\eqref{eq.pf of simplicial 1} and~\eqref{eq.pf of simplicial 2}.
Conversely, if~\eqref{eq.pf of simplicial 1} and~\eqref{eq.pf of simplicial 2} hold,
then $\abs{(\bfF^*\bfchi_\calA)(\gamma)}=\sqrt{S}$ for all $\gamma\notin\calH^\perp$ and moreover $\set{(\bfF^*\bfchi_\calA)(\gamma)}_{\gamma\in\calH^\perp,\gamma\neq1}$ is a sequence of $S$ unimodular numbers that sum to $-S$,
implying $(\bfF^*\bfchi_\calA)(\gamma)=-1$ for all $\gamma\in\calH^\perp$, $\gamma\neq1$.

(i $\Leftrightarrow$ iii)
Again, it suffices to prove that (iii) holds if and only if $\calA$ satisfies both~\eqref{eq.pf of simplicial 1} and~\eqref{eq.pf of simplicial 2}.
Here,
$\set{\bfxi_{\gamma'}}_{\gamma'\in\gamma\calH^\perp}$
is a regular $S$-simplex for each $\gamma\in\hat{\calG}$ if and only if $\abs{\ip{\bfxi_\gamma}{\bfxi_{\gamma'}}}=\frac1S$ for all $\gamma\neq\gamma'$ such that
$\overline{\gamma}=\overline{\gamma}'$.
Moreover, if a collection of regular $S$-simplices are mutually unbiased,
that is, if $\abs{\ip{\bfxi_\gamma}{\bfxi_{\gamma'}}}$ is constant over all $\gamma,\gamma'$
$\overline{\gamma}\neq\overline{\gamma}'$,
then this constant is necessarily \smash{$\frac1{\sqrt{S}}$}:
since \smash{$\set{\bfxi_\gamma}_{\gamma\in\hat{\calG}}$},
is a unit norm tight frame for \smash{$\bbC^\calA$},
its synthesis operator $\bfXi$ satisfies \smash{$\bfXi\bfXi^*=\frac{G}{S}\bfI$} and so
\begin{equation*}
\tfrac{G^2}{S}
=\Tr[(\tfrac{G^2}{S}\bfI)^2]
=\Tr[(\bfXi\bfXi^*)^2]
=\norm{\bfXi^*\bfXi}_\Fro^2
=G[(1^2)+S(\tfrac1S)^2]
+\sum_{\gamma\in\hat{\calG}}
\sum_{\substack{\gamma'\in\hat{\calG}\\\overline{\gamma}'\neq\overline{\gamma}}}
\abs{\ip{\bfxi_\gamma}{\bfxi_{\gamma'}}}^2,
\end{equation*}
implying the average value of $\abs{\ip{\bfxi_\gamma}{\bfxi_{\gamma'}}}^2$ over all $G(G-S-1)$ choices of $\gamma\neq\gamma'$, $\overline{\gamma}\neq\overline{\gamma}'$ is $\frac1S$:
\begin{equation*}
\tfrac1{G(G-S-1)}\sum_{\gamma\in\hat{\calG}}\sum_{\overline{\gamma}'\neq\overline{\gamma}}
\abs{\ip{\bfxi_\gamma}{\bfxi_{\gamma'}}}^2
=\tfrac1{G(G-S-1)}[\tfrac{G^2}{S}-\tfrac{G(S+1)}{S}]
=\tfrac1S.
\end{equation*}
In particular,
\eqref{eq.pf of simplicial 1} holds if and only if $\set{\set{\bfxi_{\gamma'}}_{\gamma'\in\gamma\calH^\perp}}_{\overline{\gamma}\in\calG/\calH^\perp}$ is a collection of $H$ mutually unbiased regular $S$-simplices.
Meanwhile~\eqref{eq.pf of simplicial 2} equates to each $\set{\bfxi_{\gamma'}}_{\gamma'\in\gamma\calH^\perp}$ summing to zero:
for any $\gamma\in\hat{\calG}$,
\begin{equation*}
\sum_{\gamma'\in\gamma\calH^\perp}\bfxi_{\gamma'}(a)
=\tfrac1{\sqrt{S}}\sum_{\gamma'\in\calH^\perp}\gamma(a)\gamma'(a)
=\tfrac1{\sqrt{S}}\gamma(a)(\bfF\bfchi_{\calH^\perp})(a)
=\tfrac1{\sqrt{S}}\gamma(a)\tfrac{G}{H}\bfchi_\calH(a)
\end{equation*}
is zero for all $a\in\calA$ if and only if $\calA$ is disjoint from $\calH$,
namely if and only if~\eqref{eq.pf of simplicial 2} holds.

For the final conclusions, now assume $\calA$ is a simplicial $\calH$-RDS,
take any $\gamma\notin\calH^\perp$, and define $\hat{\bfC}_{\gamma}$ as in the statement of the result.
When $\overline{g}=\overline{g}'$,
the fact that $\calA$ is disjoint from $\calH$ means
$\set{a\in\calA: \overline{a}=\overline{g}-\overline{g}'}
=\set{a\in\calA: a\in\calH}$ is empty,
and so every diagonal entry of $\hat{\bfC}_{\gamma}$ is $0$.
Meanwhile, if $\overline{g}=\overline{g}'$,
the fact that $\calA$ is a set of representatives of the nonidentity cosets of $\calH$ implies that $\set{a\in\calA: \overline{a}=\overline{g}-\overline{g}'}$ is a singleton set,
meaning $\hat{\bfC}_{\gamma}(\overline{g},\overline{g}')$ is unimodular, being a ``sum" of a single unimodular number.
As such, all that remains to be seen is that $\hat{\bfC}^*\hat{\bfC}=S\bfI$.
Here,
$(\hat{\bfC}_\gamma^*\hat{\bfC}_\gamma^{})(\overline{g},\overline{g}')
=(\tilde{\bfy}_\gamma*\bfy_\gamma)(\overline{g}-\overline{g}')$,
where $\bfy_\gamma\in\bbC^{\calG/\calH}$ is the first column of $\hat{\bfC}$,
defined by $\bfy_\gamma(\overline{g})$ being $0$ when $\overline{g}=\overline{0}$,
and as being $\gamma(a)$ whenever $\overline{g}\neq\overline{0}$,
where $a$ is the unique member of $\calA$ such that $\overline{a}=\overline{g}$.
It thus suffices to show that
\smash{$\tilde{\bfy}_\gamma*\bfy_\gamma=S\bfdelta_{\overline{0}}$}.
Taking Fourier transforms over the group $\calG/\calH$,
this further equates to having
\smash{$\abs{(\bfF_{\calG/\calH}^*\bfy_\gamma)(\gamma')}^2=S$} for all $\gamma'$ in the dual of $\calG/\calH$ which,
as noted in Section~2, is naturally identified with $\calH^\perp$.
This is indeed the case:
for any $\gamma'\in\calH^\perp$,
\begin{equation*}
(\bfF_{\calG/\calH}^*\bfy_\gamma)(\gamma')
=\sum_{\overline{g}\in\calG/\calH}
(\gamma')^{-1}(g)\bfy_\gamma(\overline{g})
=\sum_{a\in\calA}
(\gamma')^{-1}(a)\gamma(a)
=(\bfF^*\bfchi_\calA)(\gamma^{-1}\gamma'),
\end{equation*}
and combining this with~\eqref{eq.pf of simplicial 1} and the fact that $\gamma\notin\calH^\perp$ gives
$\abs{(\bfF_{\calG/\calH}^*\bfy_\gamma)(\gamma')}^2=S$.

In the special case where $\calA$ is an RDS that arises from a composite difference set
in the manner of Theorem~\ref{thm.composite}(a),
then for any $g\notin\calH$,
taking the unique $a\in\calA$ such that $\overline{a}=\overline{g}$ we have $\set{a'\in\calA: \overline{d}=\overline{g}}=\set{a}$ while
Theorem~\ref{thm.composite}(b) gives
$\set{d\in\calD: \overline{d}=\overline{g}}
=a+\calD_a
=a+\calB$.
As such, for any $\overline{g}\neq\overline{g}'$,
the construction of the circulant conference matrix $\bfC_\gamma$ of Theorem~\ref{thm.EITFF}(c) reduces to
\begin{equation*}
\bfC_\gamma(\overline{g},\overline{g}')
=\tfrac{S^{\frac32}}{D}\sum_{\substack{d\in\calD\\\overline{d}=\overline{g}-\overline{g}'}}\gamma(d)
=\tfrac{S^{\frac32}}{D}\sum_{b\in\calB}\gamma(a+b)
=\bigbracket{\tfrac{S^{\frac32}}{D}\bfF^*(\bfchi_\calB)(\gamma^{-1})}\gamma(a)
=z\hat{\bfC}_\gamma(\overline{g},\overline{g}')
\end{equation*}
where \smash{$z=\bigbracket{\tfrac{S^{\frac32}}{D}\bfF^*(\bfchi_\calB)(\gamma^{-1})}$} is constant over all $\overline{g}\neq\overline{g}'$.
Since every diagonal entry of both $\bfC_\gamma$ and $\hat{\bfC}_\gamma$ is zero, this implies $\bfC_\gamma=z\hat{\bfC}_\gamma$ where $\abs{z}=1$ by Lemma~\ref{lem.small difference sets}.
\end{proof}

\subsection{Constructions of fine difference sets, composite difference sets and amalgams}

We now discuss how the known fine difference sets from~\cite{FickusJKM18} fit into the framework discussed here.
We shall see that some fine difference sets are amalgams while others are not, and that some amalgams are composite difference sets while others are not.

\subsubsection{Singer difference sets}

For any prime power $Q$, let $\bbF_Q$ denote the finite field of order $Q$ and let $\bbF_Q^\times$ denote its multiplicative group, which is well known to be cyclic.
For any integer $J\geq2$,
let $\tr_{Q^J/Q}:\bbF_{Q^J}\rightarrow\bbF_{Q}$,
\smash{$\tr_{Q^J/Q}(x):=\sum_{j=0}^{J-1}x^{Q^j}$} be the \textit{field trace},
which is a well-known, nontrivial linear functional of \smash{$\bbF_{Q^J}$}, regarded as a $J$-dimensional vector space over the field $\bbF_Q$.
In this setting, the affine hyperplane $\calE=\set{x\in\bbF_{Q^J}^\times : \tr_{Q^J/Q}(x)=1}$
is a well-known RDS for the cyclic group \smash{$\calG=\bbF_{Q^J}^\times$} of order \smash{$Q^J-1$}~\cite{Pott95}.
Here, the sets \smash{$\set{x\calE: x\in\bbF_{Q^J}^\times}$} are the distinct affine hyperplanes of~\smash{$\bbF_{Q^J}$} that do not contain $0$.
The affine hyperplanes $\calE$ and $x\calE$ are equal if and only if $x=1$,
and are parallel if and only if
\smash{$x\in\bbF_{Q}^\times$}, $x\neq1$.
In any other case, these two affine hyperplanes intersect in an affine subspace of codimension $2$.
Thus, for any $x\in\bbF_{Q^J}^\times$,
\begin{equation*}
\#[\calE\cap(x\calE)]
=\left\{\begin{array}{cl}
Q^{J-1},&\ x=1,\smallskip\\
0,      &\ x\in\bbF_Q^\times,\ x\neq1,\smallskip\\
Q^{J-2},&\ x\notin\bbF_Q^\times.
\end{array}\right.
\end{equation*}
As such, letting $\calK=\bbF_Q^\times$,
$\calE$ is a \smash{$\calK$-$\RDS(\frac{Q^J-1}{Q-1},Q-1,Q^{J-1},Q^{J-2})$} for $\calG$.
Quotienting by $\calK$ thus produces a \smash{$Q^{J-1}$}-element difference set
\smash{$\overline{\calE}
=\set{\overline{x}\in\bbF_{Q^J}^\times/\bbF_Q^\times: \tr_{Q^J/Q}(x)\neq0}$}
for \smash{$\calG/\calK=\bbF_{Q^J}^\times/\bbF_Q^\times$};
the complement of $\overline{\calE}$ is the classical Singer difference set.
When $J\geq4$ is even, a shift of this difference set is known to be fine~\cite{FickusJKM18}.
Below, we show that this fine difference set is in fact composite,
and the RDS from which it arises is part of a larger family of RDSs to which Theorem~\ref{thm.simplicial} applies.

To elaborate, in the special case where $J=2$,
the aforementioned construction reduces to
\begin{equation}
\label{eq.Singer complment 0}
\calE=\set{x\in\bbF_{Q^2}^\times : \tr_{Q^2/Q}(x)=x+x^Q=1}
\end{equation}
being a $\calK$-$\RDS(Q+1,Q-1,Q,1)$ for \smash{$\calG=\bbF_{Q^2}^\times$} where \smash{$\calK=\bbF_Q^\times$}.
Since quotienting $\calE$ by $\calK$ produces a $Q$-element difference set $\overline{D}$ for the group $\calG/\calK$ of order $Q+1$,
we can always shift $\calE$ if necessary so that its quotient avoids $\set{\overline{0}}$,
that is, so that $\calE$ is disjoint from $\calK$.
In fact, when $Q$ is even, no shift is necessary:
every $x\in\bbF_Q$ satisfies $x^Q=x$ and so $\tr_{Q^2/Q}(x)=x+x^Q=x+x=0$,
implying $\calK$ is disjoint from $\calE$.
Moreover, when $Q$ is odd, such a shift can be computed explicitly:
letting $\alpha$ be a generator of $\bbF_{Q^2}^\times$,
\smash{$\beta=\alpha^{-(Q+1)/2}$} satisfies \smash{$\beta^{Q-1}=\alpha^{-(Q^2-1)/2}=-1$};
as such, \smash{$\tr_{Q^2/Q}(\beta x)=(\beta x)+(\beta x)^Q=\beta (x-x^Q)$} for all \smash{$x\in\bbF_{Q^2}^\times$}
implying
\begin{equation*}
\alpha^{(Q+1)/2}\calE
=\beta^{-1}\calE
=\set{x\in\bbF_{Q^2}^\times : \tr_{Q^2/Q}(\beta x)=1}
=\set{x\in\bbF_{Q_2}^\times : x-x^Q=\beta^{-1}}
\end{equation*}
is a $\calK$-RDS for $\calG$ that is disjoint from
\smash{$\calK=\bbF_Q^\times=\set{x\in\bbF_{Q^2}^\times: x-x^Q=0}$}.

Regardless, for any prime power $Q$,
we see that there exists a $\calK$-$\RDS(Q+1,Q-1,Q,1)$ that is disjoint from $\calK$,
namely an RDS that is simplicial in the sense of Definition~\ref{def.simplicial}.
By Theorem~\ref{thm.simplicial}(iii),
the corresponding $(Q^2-1)$-vector harmonic tight frame is a union of $Q+1$ mutually unbiased regular $Q$-simplices.
Theorem~\ref{thm.simplicial} also yields a circulant conference matrix of size $Q+1$.
As we next explain, some quotients of some RDSs of this type naturally arise from composite difference sets in the manner of Theorem~\ref{thm.composite}.

Here, for any prime power $Q$ and $J\geq2$,
we consider the Singer-complement difference set that arises from an affine hyperplane in a $2J$-dimensional extension of $\bbF_Q$,
namely
\begin{equation}
\label{eq.Singer complment 1}
\calD
=\set{\overline{x}\in\bbF_{Q^{2J}}^\times/\bbF_Q^\times: \tr_{Q^{2J}/Q}(x)\neq0}
\end{equation}
where $\tr_{Q^{2J}/Q}(x)=\sum_{j=0}^{2J-1}x^{Q^j}$.
Such difference sets constitute ``half" of all Singer-complement difference sets.
As noted in~\cite{FickusJKM18}, the remaining ``half" of these difference sets---those  that arise in odd-dimensional extensions of $\bbF_Q$---do not seem to be fine in general,
and in fact cannot be fine when $Q$ is an odd power of a prime since in such cases $S$ is not an integer.
From above, we know that $\calD$ is a difference set of cardinality $D=Q^{2J-1}$ for the cyclic group $\calG=\bbF_{Q^{2J}}^\times/\bbF_Q^\times$ of order \smash{$G=\frac{Q^{2J}-1}{Q-1}$}.
It follows that
\begin{equation*}
G-1=Q(\tfrac{Q^{2J-1}-1}{Q-1}),
\quad
G-D=\tfrac{Q^{2J-1}-1}{Q-1},
\quad
S:=[\tfrac{D(G-1)}{G-D}]^{\frac12}=Q^J,
\quad
H:=\tfrac{G}{S+1}=\tfrac{Q^J-1}{Q-1}.
\end{equation*}
For $\calD$ to be fine,
it must be disjoint from a subgroup $\calH$ of $\calG$ of order $H$.
(To be precise, we shall see that there is always a shift of~\eqref{eq.Singer complment 1} for which this occurs.)
Since $\calG$ is cyclic and $S+1$ divides $G$, there is exactly one such subgroup, namely \smash{$\calH=\bbF_{Q^J}^\times/\bbF_Q^\times$}.
If $\calD$ is composite,
Theorem~\ref{thm.composite}(a) implies that it factors in terms of an $\calH$-RDS with parameters
\smash{$(S+1,H,S,\tfrac{S-1}{H})=(Q^J+1,\tfrac{Q^J-1}{Q-1},Q^J,Q-1)$}.
There is a natural candidate for such an RDS:
taking ``Q" in~\eqref{eq.Singer complment 0} to be $Q^J$
gives that
\begin{equation}
\label{eq.Singer complment 2}
\set{x\in\bbF_{Q^{2J}}^\times : \tr_{Q^{2J}/Q^J}(x)=1}
\end{equation}
is a $\calK$-$\RDS(Q^J+1,Q^J-1,Q^J,1)$ where \smash{$\calK=\bbF_{Q^J}^\times$};
quotienting this by \smash{$\bbF_{Q}^\times$} produces an $\calH$-RDS with the desired parameters, namely
\begin{equation*}
\calA
=\set{\overline{x}\in\bbF_{Q^{2J}}^\times/\bbF_{Q}^\times :
\ \exists\ y\in x\bbF_{Q}^\times\text{ s.t. }\tr_{Q^{2J}/Q^J}(y)=1}
=\set{\overline{x}\in\bbF_{Q^{2J}}^\times/\bbF_{Q}^\times : \tr_{Q^{2J}/Q^J}(x)\in\bbF_Q^\times}.
\end{equation*}
Further recall from Theorem~\ref{thm.composite} that when $\calD$ is composite, we expect $\calD_a=\calH\cap(a^{-1}\calD)=\calB$ for all $a\in\calA$.
Here,
for any $a\in\calA$,
writing $a=\overline{y}$ where $y\in\bbF_{Q^{2J}}^\times$ satisfies $\tr_{Q^{2J}/Q^J}(y)\in\bbF_Q^\times$, we have
\begin{align}
\nonumber
\calD_{a}
&=(\bbF_{Q^J}^\times/\bbF_Q^\times)
\cap
\set{\overline{y}^{-1}\overline{x}\in\bbF_{Q^{2J}}^\times/\bbF_Q^\times: \tr_{Q^{2J}/Q}(x)\neq0}\\
\nonumber
&=\set{\overline{y^{-1}x}\in\bbF_{Q^{J}}^\times/\bbF_Q^\times: \tr_{Q^{2J}/Q}(x)\neq0}\\
\label{eq.Singer complment 3}
&=\set{\overline{z}\in\bbF_{Q^{J}}^\times/\bbF_Q^\times: \tr_{Q^{2J}/Q}(yz)\neq0}.
\end{align}
To continue simplifying this,
we recall from finite field theory that finite field traces factor over intermediate fields: the \textit{freshman's dream} implies that for any $x\in\bbF^{Q^{2J}}$
\begin{equation*}
\tr_{Q^J/Q}(\tr_{Q^{2J}/Q^J}(x))
=\tr_{Q^J/Q}(1+x^{Q^J})
=\sum_{j=0}^{J-1}(1+x^{Q^J})^{Q^j}
=\sum_{j=0}^{J-1}(1+x^{Q^{J+j}})
=\tr_{Q^{2J}/Q^J}(x).
\end{equation*}
When combined with the fact that $\tr_{Q^{2J}/Q^J}(y)\in\bbF_Q^\times$,
and the fact that $\tr_{Q^{2J}/Q^J}$ is linear in $Q^J$ while $\tr_{Q^J/Q}$ is linear in $\bbF_Q$,
this implies that for any representative $z\in\bbF_{Q^J}^\times$ of a
coset $\overline{z}\in\bbF_{Q^{J}}^\times/\bbF_Q^\times$,
\begin{equation*}
\tr_{Q^{2J}/Q}(yz)
=\tr_{Q^J/Q}[\tr_{Q^{2J}/Q^J}(yz)]
=\tr_{Q^J/Q}[z\tr_{Q^{2J}/Q^J}(y)]
=[\tr_{Q^{2J}/Q^J}(y)][\tr_{Q^J/Q}(z)].
\end{equation*}
When taken together with the fact that $\tr_{Q^{2J}/Q^J}(y)\neq0$,
this simplifies~\eqref{eq.Singer complment 3} as
\begin{equation*}
\calD_{a}
=\set{\overline{z}\in\bbF_{Q^{J}}^\times/\bbF_Q^\times: [\tr_{Q^{2J}/Q^J}(y)][\tr_{Q^J/Q}(z)]\neq0}
=\set{\overline{z}\in\bbF_{Q^{J}}^\times/\bbF_Q^\times: \tr_{Q^J/Q}(z)\neq0}.
\end{equation*}
That is, for every $a\in\calA$, we have $\calD_a=\calB$
where \smash{$\calB=\set{\overline{z}\in\bbF_{Q^{J}}^\times/\bbF_Q^\times: \tr_{Q^J/Q}(z)\neq0}$} is, by definition, the complement of the canonical Singer difference set in the group $\calH=\bbF_{Q^J}^\times/\bbF_Q^\times$.
Here, the fact that $\calA$ is an $\calH$-RDS for $\calG$ while $\calB$ is a subset of $\calH$ implies that the sets $\set{a\calB}_{a\in\calA}$ are disjoint.
As such, for each $a\in\calA$ we have $a\calB=a\calD_a=\calD\cap(a\calH)\subseteq\calD$.
Thus, $\calD$ contains the disjoint union $\sqcup_{a\in\calA}a\calB$.
Moreover, as the cardinality of $\sqcup_{a\in\calA}a\calB$ is
$\#(\calA)\#(\calB)=Q^J(Q^{J-1})=Q^{2J-1}=\#(\calD)$,
$\calD$ is this disjoint union, meaning $\bfchi_\calD=\bfchi_\calA*\bfchi_\calB$ where $\calA$ has cardinality $S=Q^J$.

Comparing this against Definition~\ref{def.composite},
all that remains to be shown is that $\calD$ is fine, namely that it is disjoint from $\calH=\bbF_{Q^J}^\times/\bbF_Q^\times$.
Though this is not always the case, it is always possible to shift $\calD$ so as to gain this property without losing the others.
In fact, since $\bfchi_\calD=\bfchi_\calA*\bfchi_\calB$ where $\calB$ is a subset of $\calH$ and where $\calA$ is obtained by quotienting~\eqref{eq.Singer complment 2} by $\bbF_Q^\times$,
it suffices to shift~\eqref{eq.Singer complment 2} so that it becomes disjoint from \smash{$\bbF_{Q^J}^\times$}.
We have already seen how to do this:
no shift is necessary when $Q$ is even,
and when $Q$ is odd, we can multiply~\eqref{eq.Singer complment 2} by \smash{$\alpha^{(Q^J+1)/2}$} where $\alpha$ is a generator of \smash{$\bbF_{Q^{2J}}^\times$}.
We summarize these facts as follows:

\begin{theorem}
\label{thm.Singer}
For any prime power $Q$,
the classical RDS~\eqref{eq.Singer complment 0} can be shifted to produce a simplicial $\RDS(Q+1,Q-1,Q,1)$.
Moreover, for any
$J\geq 2$, the complement of the Singer difference set in the cyclic group $\calG$ of order \smash{$\frac{Q^{2J}-1}{Q-1}$} can be shifted so as to produce a
$Q^{2J-1}$-element composite difference set $\calD$ for $\calG$,
and the resulting factors $\calA$ and $\calB$ are a simplicial
$\RDS(Q^J+1,\tfrac{Q^J-1}{Q-1},Q^J,Q-1)$ and the complement of the Singer difference set for the \smash{$\frac{Q^J-1}{Q-1}$}-element subgroup $\calH$ of $\calG$, respectively.
\end{theorem}

For any prime power $Q$ and $J\geq2$,
applying Theorem~\ref{thm.composite} to these composite difference sets yields an EITFF whose orthogonal projection operators satisfy~\eqref{eq.composite projections}, and also recovers the underlying \smash{$\RDS(Q^J+1,\tfrac{Q^J-1}{Q-1},Q^J,Q-1)$}.
Applying Theorem~\ref{thm.simplicial} to these RDSs produces \smash{$\tfrac{Q^J-1}{Q-1}$} mutually unbiased regular $Q^J$-simplices as well as circulant conference matrices of size $Q^J+1$.
In the special case where $Q=2$, $J=2$, this construction yields the composite difference set for $\bbF_{2^4}^\times\cong\bbZ_{15}$ given in Example~\ref{ex.8x15 fine},
and the resulting conference matrices are given in Example~\ref{ex.8x15 embeddings}.
Meanwhile, applying Theorem~\ref{thm.simplicial} directly to the (suitably shifted) version of the $\RDS(Q+1,Q-1,Q,1)$ given in~\eqref{eq.Singer complment 0} gives $Q-1$ mutually unbiased regular $Q$-simplices as well as circulant conference matrices of size $Q+1$.
This latter construction is more general, as it is the only one that yields circulant conference matrices of size $P+1$ where $P$ is prime.
We note that a harmonic tight frame appearing from this $\RDS(Q+1,Q-1,Q,1)$ has recently appeared elsewhere in the frame literature:
concatenating it with the standard basis yields tight frames that achieve the \textit{orthoplex bound}, and also provide an alternative solution to a reconstruction problem in quantum information theory that is usually solved with mutually unbiased bases~\cite{BodmannH16}.

The fact that complements of Singer difference sets factor in this way is not new.
Indeed, it is the fundamental idea behind the now-classical method of Gordon, Mills and Welch for producing many nonequivalent difference sets with Singer-complement parameters~\cite{GordonMW62,Pott95}.
In fact, every known example of an RDS either quotients to the entire group or has parameters that match those of the complement of a Singer difference set~\cite{Pott95},
namely \smash{$(\frac{G}{H},H,D,\Lambda)$} where \smash{$\tfrac{G}{H}=\tfrac{Q^J-1}{Q-1}$}, $D=Q^{J-1}$ and $H\Lambda=Q^{J-2}(Q-1)$.
In light of Theorem~\ref{thm.composite}(a), it is therefore not too surprising that these are the only composite difference sets we have discovered so far.

\subsubsection{Twin prime power difference sets}

For any odd prime power $Q$,
let \smash{$\calS_Q:=\set{x^2: x\in\bbF_Q^\times}$} and let
\smash{$\calN_Q:=\bbF_Q^\times\backslash\calS_Q$} be the nonzero squares and nonsquares in $\bbF_Q$, with both sets having cardinality $\frac12(Q-1)$.
When $Q$ and $Q+2$ are both powers of odd primes,
the set
\begin{equation}
\label{eq.TPP complement}
\calD
=(\set{0}\times\bbF_{Q+2}^\times)
\sqcup(\calS_Q\times\calN_{Q+2})
\sqcup(\calN_Q\times\calS_{Q+2})
\end{equation}
is a difference set for $\calG=\bbF_Q\times\bbF_{Q+2}$ of cardinality $D=Q+1+2\frac14(Q-1)(Q+1)=\tfrac12(Q+1)^2$,
being the complement of $(\bbF_Q\times\set{0})
\sqcup(\calS_Q\times\calS_{Q+2})
\sqcup(\calN_Q\times\calN_{Q+2})$,
which is the well-known twin prime power difference set for $\calG$.
Here $G=Q(Q+2)$ and so
\begin{equation*}
G-1=Q^2+2Q-1,
\quad
G-D
=\tfrac12(Q^2+2Q-1),
\quad
S:=[\begin{array}{cl}\tfrac{D(G-1)}{G-D}\end{array}]^{\frac12}=Q+1,
\quad
H:=\tfrac{G}{S+1}=Q,
\end{equation*}
and so $\calD$ is fine, being disjoint from the subgroup $\calH=\bbF_Q\times\set{0}$ of order $H$~\cite{FickusJKM18}.
By Theorem~\ref{thm.fine}, every nonidentity coset of $\calH$ thus intersects $\calD$ in exactly $\frac{D}{S}=\frac12(Q+1)$ points.
This fine difference set can only be an amalgam if the necessary condition of Theorem~\ref{thm.EITFF}(b) is met,
namely only if $S^3=(Q+1)^3$ divides $D^2=\tfrac14(Q+1)^4$,
that is, only if $Q\equiv 3\bmod 4$.
Here, it is notable that \smash{$\frac{D-1}{G-D}=1$} regardless of whether $Q$ is congruent to $1$ or $3$ modulo $4$.
That is, even though every composite difference set is an amalgam,
there are at least some cases in which the necessary condition on composite difference sets given in Theorem~\ref{thm.composite}(a) does not imply the necessary condition on amalgams given in Theorem~\ref{thm.EITFF}(b).
Put another way, in order for a composite difference set to exist,
both $\smash{\frac{D^2}{S^3}}$ and \smash{$\frac{D-1}{G-D}$} are necessarily integers.

When $Q\equiv 3\bmod4$, we claim that $\calD$ is an amalgam.
Since $\calD$ is disjoint from $\calH$,
it suffices to show that for all $g\notin\calH$,
$\calD_g=\calH\cap(\calD-g)$ is a difference set for $\calH$.
Here, any $g\notin\calH=\bbF_Q\times\set{0}$ is of the form $g=(x,y)$ where $y\neq0$ and so
\begin{equation*}
\calD-g
=
\bigbracket{\set{-x}\times(\bbF_{Q+2}^\times-y)}
\sqcup\bigbracket{(\calS_Q-x)\times(\calN_{Q+2}-y)}
\sqcup\bigbracket{(\calN_Q-x)\times(\calS_{Q+2}-y)}.
\end{equation*}
When $y\in\calS_{Q+2}$,
the intersection of $\calD-g$ with $\calH=\bbF_Q\times\set{0}$ is thus
\begin{equation*}
\calD_g
=(\set{-x}\times\set{0})\sqcup\bigbracket{(\calN_Q-x)\times\set{0}}
=(\bbF_Q\backslash\calS_Q-x)\times\set{0},
\end{equation*}
which is a difference set for $\calH$,
since $\bbF_Q\backslash\calS_Q-x$ is a difference set for $\bbF_Q$,
being a shift of the complement of the difference set $\calS_Q$.
Similarly, when $y\in\calN_{Q+2}$,
\begin{equation*}
\calD_g
=(\set{-x}\times\set{0})\sqcup\bigbracket{(\calS_Q-x)\times\set{0}}
=(\bbF_Q\backslash\calN_Q-x)\times\set{0},
\end{equation*}
which is also a difference set for $\calH$.
Overall, we see that $\calD$ is an amalgam when $Q\equiv3\bmod4$.

However, as we now explain,
$\calD$ is not composite in general since $\bbF_Q\backslash\calS_Q$ and $\bbF_Q\backslash\calN_Q$ are only shifts of each other when $Q=3$.
Indeed, when $Q=3$,
$\bbF_3\backslash\calS_3=\set{0,2}$ and
$\bbF_3\backslash\calN_3=\set{0,1}$ are shifts of each other,
and so in this case,
\eqref{eq.TPP complement} becomes the composite difference set
\begin{equation*}
\set{(0,1),(0,2),(0,3),(0,4)}\sqcup\set{(1,2),(1,3)}\sqcup\set{(2,1),(2,4)}
\end{equation*}
for $\bbZ_3\times\bbZ_5$.
Inverting the isomorphism $1\mapsto(1,1)$ gives an alternative construction of the composite difference set
$\set{6,12,3,9,7,13,11,14}$ for $\bbZ_{15}$ given in Example~\ref{ex.8x15 fine}.

Meanwhile, for any $Q\neq 3$, $\calS_Q$ and $\calN_Q$ are not shifts of each other meaning~\eqref{eq.TPP complement} is an amalgam that is not a composite difference set.
For an elementary proof of this fact, take any prime power $Q$ and suppose that there exists $x\in\bbF_Q$ such that $\calN_Q=\calS_Q+x$.
Here, $\calS_Q$ and $\calN_Q$ partition $\bbF_Q^\times$ into two sets of cardinality $\frac12(Q-1)$.
Specifically, $\calS_Q$ and $\calN_Q$ are the even and odd powers, respectively, of any generator $\alpha$ of \smash{$\bbF_Q^\times$}.
As such, we cannot have $x=0$, and moreover if $x\in\calN_Q$ then dividing $\calN_Q=\calS_Q+x$ by $x$ gives $\calS_Q=\calN_Q+1$; since $1=1^2\in\calS_Q$, this implies $0\in\calS_Q-1=\calN_Q$,
a contradiction.
Thus, $x\in\calS_Q$, and dividing $\calN_Q=\calS_Q+x$ by $x$ gives $\calN_Q=\calS_Q+1$.
This fact is thus also true for any divisor $Q'$ of $Q$:
when \smash{$\bbF_{Q'}^\times\subseteq\bbF_Q^\times$} we have $\calS_{Q'}\subseteq\calS_Q$ and so $\calS_{Q'}+1\subseteq\calS_Q+1=\calN_Q\subseteq\calN_{Q'}$;
since both $\calS_{Q'}$ and $\calN_{Q'}$ have cardinality $\frac12(Q'-1)$, this implies $\calN_{Q'}=\calS_{Q'}+1$.
This in turn implies that $Q$ is prime:
if not, letting $Q'=P^2$ where $Q=P^J$ for some $J\geq 2$,
we have $P^2\equiv1\bmod 4$,
and so $-1=\beta^{(P^2-1)/2}=(\beta^{(P^2-1)/4})^2$ for any generator $\beta$ of $\bbF_{P^2}^\times$;
thus $0\in\calS_{P^2}+1=\calN_{P^2}$, a contradiction.
Moreover, when $Q=P$ is prime,
we necessarily have that $P\equiv 3\bmod4$ or else $-1\in\calS_P$,
implying $0\in\calN_P$, a contradiction.
To summarize our progress so far, if $\calS_Q$ and $\calN_Q$ are shifts of each other,
then $Q$ is necessarily some prime $P\equiv 3\bmod 4$, and $\calN_P=\calS_P+1$.
In particular, this implies that every square modulo $P$ is followed by a nonsquare modulo $P$.
Since $\calS_P$ and $\calN_P$ partition $\bbF_P^\times=\set{1,2,\dotsc,P-1}$,
this is only possible if $\calS_P=\set{1,3,\dotsc,P-2}$ and $\calN_P=\set{2,4,\dotsc,P-1}$ are the odd and even numbers modulo $P$, respectively.
As we have already seen, this all indeed happens in the special case where $P=3$.
However, it fails for all greater primes since $4=2^2$ is a square.

We summarize these facts as follows:
\begin{theorem}
\label{thm.TPP}
Let $Q$ and $Q+2$ be odd twin prime powers,
and let $\calD$ be the $\frac12(Q+1)^2$ element subset~\eqref{eq.TPP complement} of $\calG=\bbF_Q\times\bbF_{Q+2}$,
namely the complement of the classical twin prime power difference set.
Then $\calD$ is a fine difference set, and is an amalgam if and only if $Q\equiv 3\bmod 4$.
Moreover, $\calD$ is a composite difference set only when $Q=3$.
\end{theorem}

Applying Theorem~\ref{thm.EITFF}(c) to these amalgams produces circulant conference matrices of size $S+1=Q+2$.
These conference matrices are distinct from those that arise from any known simplicial RDS via Theorem~\ref{thm.simplicial} since, as mentioned above, all of those are of size $Q+1$ where $Q$ is a prime power.
For example, when $Q=11$, applying Theorem~\ref{thm.EITFF}(c) to~\eqref{eq.TPP complement} gives, to our knowledge, the only known construction of a $13\times 13$ circulant conference matrix;
since $12$ is not a prime power, the requisite RDS needed to apply Theorem~\ref{thm.simplicial} to produce such a matrix is not known to exist~\cite{Pott95}.
These difference sets also provide our only known construction of amalgams that are not cyclic: when $Q=27$, for example, \eqref{eq.TPP complement} is an amalgam for the group $\bbF_{27}\times\bbF_{29}\cong\bbZ_3\times\bbZ_3\times\bbZ_{87}$.
We further note that in the case where $Q=11$,
the fact that $\calD$ is not composite means we should not expect the orthogonal projection operators onto the subspaces that constitute the corresponding EITFF to satisfy the property given in Theorem~\ref{thm.composite}(b);
an explicit computation in this case reveals that they indeed do not have this property.

\subsubsection{McFarland difference sets}

For any prime power $Q$ and integer $J\geq 2$,
regard $\bbF_Q^J$ as a $J$-dimensional vector space over the finite field $\bbF_Q$ of order $Q$.
Let $\calK$ be any abelian group of order \smash{$\frac{Q^J-1}{Q-1}+1$},
and let $\set{\calV_k}_{k\in\calK,\,k\neq0}$ be any enumeration of the distinct hyperplanes of \smash{$\bbF_Q^J$} according to the nonzero members of $\calH$.
The set \smash{$\calD=\sqcup_{k\in\calK,\,k\neq0}\set{(k,v): v\in\calV_k}$}
is then a \textit{McFarland} difference set for the group \smash{$\calG=\calK\times\bbF_Q^J$}.
Here, a direct calculation reveals
\begin{equation*}
D=Q^{J-1}(\tfrac{Q^J-1}{Q-1}),
\quad
G=Q^J(\tfrac{Q^J-1}{Q-1}+1),
\quad
%G-1=(Q^J+Q-1)\tfrac{Q^J-1}{Q-1},
%\quad
%G-D=Q^{J-1}(Q^J+Q-1),
%\quad
S:=[\tfrac{D(G-1)}{G-D}]^{\frac12}=\tfrac{Q^J-1}{Q-1},
\quad
H:=\tfrac{G}{S+1}=Q^J.
\end{equation*}
As noted in~\cite{FickusJKM18}, every such $\calD$ is fine,
being disjoint from a subgroup of order \smash{$H=Q^J$},
namely $\calH=\set{0}\times\bbF_Q^J$.
By Theorem~\ref{thm.fine}, every coset of $\calH$ intersects $\calD$ in
\smash{$\frac{D}{S}=Q^{J-1}$} points;
by inspection, these intersections are of the form $\set{k}\times\calV_k$ for some $k\in\calK$, $k\neq0$.
However,
\begin{equation*}
\tfrac{D^2}{S^3}
=\tfrac{Q^{2J-2}(Q-1)}{Q^J-1}
=Q^{J-2}(Q-1)+\tfrac{Q^{J-2}(Q-1)}{Q^J-1}
\end{equation*}
is never an integer since $0<Q^{J-2}(Q-1)<Q^J-1$.
As such, any such difference sets are never amalgams, and so are never composite.

That said, even in this case, much of the machinery developed to prove Theorem~\ref{thm.EITFF} still provides insights into these fine difference sets that go beyond those provided by the techniques of~\cite{FickusJKM18}.
For instance, when $Q=2$ and $J=2$,
$\calD=\set{1000,1001,0100,0110,1100,1111}$ is a $6$-element McFarland difference set for the group $\calG=\bbZ_2^4$ of order $16$.
%Every nonzero $g\in\calG$ appears exactly twice in the corresponding difference table:
%\begin{equation*}
%\begin{array}{c|cccccc}
%-   &1000&1001&0100&0110&1100&1111\\\hline
%1000&0000&0001&1100&1110&0100&0111\\
%1001&0001&0000&1101&1111&0101&0110\\
%0100&1100&1101&0000&0010&1000&1011\\
%0110&1110&1111&0010&0000&1010&1001\\
%1100&0100&0101&1000&1010&0000&0011\\
%1111&0111&0110&1011&1001&0011&0000
%\end{array}.
%\end{equation*}
Here, $D=6$, $G=16$, $S=3$,
and $\calD$ is disjoint from the subgroup $\calH=\set{0000,0010,0001,0011}$ of order
\smash{$H=\frac{G}{S+1}=4$}.
We identify $\hat{\calG}$ with $\bbZ_2^4$,
regarding $n_1n_2n_3n_4$ as the character
$g_1g_2g_3g_4\mapsto(-1)^{g_1n_1+g_2n_2+g_3n_3+g_4n_4}$.
In particular, $\calH^\perp$ is identified with those $n_1n_2n_3n_4$ such that $(-1)^{n_3}=(-1)^{n_4}=1$,
namely $\set{0000,1000,0100,1100}$.
Here, to form the corresponding $16\times 16$ character table, we elect to order these characters lexicographically,
that is,
as $\set{0000,1000,0100,1100,0010,\dotsc,1111}$.
Under these arbitrarily chosen orderings of $\calD$ and $\hat{\calG}$,
the synthesis operator of the corresponding harmonic $\ETF(6,16)$,
formed by extracting the $6$ rows of the character table that correspond to $\calD$ and then normalizing columns, is:
\begin{equation*}
\bfPhi=\frac1{\sqrt{6}}\left[\begin{array}{rrrrrrrrrrrrrrrr}
1&-1& 1&-1& 1&-1& 1&-1& 1&-1& 1&-1& 1&-1& 1&-1\\
1&-1& 1&-1& 1&-1& 1&-1&-1& 1&-1& 1&-1& 1&-1& 1\\
1& 1&-1&-1& 1& 1&-1&-1& 1& 1&-1&-1& 1& 1&-1&-1\\
1& 1&-1&-1&-1&-1& 1& 1& 1& 1&-1&-1&-1&-1& 1& 1\\
1&-1&-1& 1& 1&-1&-1& 1& 1&-1&-1& 1& 1&-1&-1& 1\\
1&-1&-1& 1&-1& 1& 1&-1&-1& 1& 1&-1& 1&-1&-1& 1
\end{array}\right].
\end{equation*}
Here, for any $n_1n_2n_3n_4$ in $\hat{\calG}$,
the corresponding $6\times 4$ submatrix~\eqref{eq.def of Phi_gamma} is
\begin{equation*}
\bfPhi_{n_1n_2n_3n_4}
=\left[\begin{array}{cccc}
\bfphi_{n_1n_2n_3n_4}&
\bfphi_{(n_1+1)n_2n_3n_4}&
\bfphi_{n_1(n_2+1)n_3n_4}&
\bfphi_{(n_1+1)(n_2+1)n_3n_4}
\end{array}\right].
\end{equation*}
As such, the above matrix $\bfPhi$ can be regarded as the concatenation of $\bfPhi_{0000}$, $\bfPhi_{0010}$, $\bfPhi_{0001}$ and $\bfPhi_{0011}$.
Meanwhile, the matrix $\bfPsi$ in~\eqref{eq.def of Psi} is obtained by removing the first row of the character table of $\calG/\calH\cong\bbZ_2\times\bbZ_2$, giving the synthesis operator of a particularly nice tetrahedron:
\begin{equation*}
\bfPsi=\frac1{\sqrt{3}}\left[\begin{array}{rrrr}
1&-1& 1&-1\\
1& 1&-1&-1\\
1&-1&-1& 1\\
\end{array}\right].
\end{equation*}
Theorem~\ref{thm.EITFF}(a) then gives $\bfPhi_{n_1n_2n_3n_4}=\bfE_{n_1n_2n_3n_4}\bfPsi$ where
\begin{equation*}
\bfE_{0000}=\left[\begin{array}{rrr}
 1& 0& 0\\
 1& 0& 0\\
 0&\phantom{-}1& 0\\
 0& 1& 0\\
 0& 0&\phantom{-}1\\
 0& 0& 1
\end{array}\right],\
\bfE_{0010}=\left[\begin{array}{rrr}
 1& 0& 0\\
 1& 0& 0\\
 0& 1& 0\\
 0&-1& 0\\
 0& 0& 1\\
 0& 0&-1
\end{array}\right],\
\bfE_{0001}=\left[\begin{array}{rrr}
 1& 0& 0\\
-1& 0& 0\\
 0&\phantom{-}1& 0\\
 0& 1& 0\\
 0& 0& 1\\
 0& 0&-1
\end{array}\right],\
\bfE_{0011}=\left[\begin{array}{rrr}
 1& 0& 0\\
-1& 0& 0\\
 0& 1& 0\\
 0&-1& 0\\
 0& 0&\phantom{-}1\\
 0& 0& 1
\end{array}\right].
\end{equation*}
As such, our $\ETF(6,16)$ is comprised of four tetrahedra,
each embedded in a $3$-dimensional subspace
$\calU_{n_3n_4}=\calU_{\overline{n_1n_2n_3n_4}}=\rmC(\bfE_{n_1n_2n_3n_4})$ of $\bbC^{\calD}$.
Moreover, Theorem~\ref{thm.EITFF}(a) implies every corresponding cross-Gram matrix is diagonal.
In fact, by inspection, the cross-Gram matrix of any pair of the above four isometries is one of the following three matrices:
\begin{equation*}
\left[\begin{array}{ccc}
1&0&0\\
0&0&0\\
0&0&0
\end{array}\right],\quad
\left[\begin{array}{ccc}
0&0&0\\
0&1&0\\
0&0&0
\end{array}\right],\quad
\left[\begin{array}{ccc}
0&0&0\\
0&0&0\\
0&0&1
\end{array}\right].
\end{equation*}
From~\cite{FickusJKM18},
we already knew that subspaces $\set{\calU_{00},\calU_{10},\calU_{01},\calU_{11}}$ are equi-chordal,
implying all of these cross-Gram matrices have the same Frobenius norm.
Here, we can say more:
taking the absolute values of these diagonal entries of these cross-Gram matrices reveals that the principal angles between any pair of these subspaces are $\set{0,\frac{\pi}2,\frac{\pi}{2}}$,
meaning in particular that any pair of these subspaces intersect in a line.
This is a hallmark feature of the subspaces spanned by the regular simplices that comprise a Steiner ETF.
In fact, it is known that harmonic ETFs arising from McFarland difference sets are unitarily equivalent to certain Steiner ETFs arising from affine geometries~\cite{JasperMF14}.
What is remarkable is that, even if this fact was not already known,
the machinery of Theorem~\ref{thm.EITFF} would have naturally led one to this realization.

%%%%%%%%%%%%%%%%%%%%%%%%%%%%%%%%%%%%%%%%%%%%%%%%%%%%%%%%%%%%%%%%
\section{Conclusions}
%%%%%%%%%%%%%%%%%%%%%%%%%%%%%%%%%%%%%%%%%%%%%%%%%%%%%%%%%%%%%%%%

We now have three, increasingly nice types of difference sets, namely
fine difference sets, amalgams and composite difference sets, as given in Definitions~\ref{def.fine}, \ref{def.amalgam} and~\ref{def.composite}, respectively.
Every composite difference set is an amalgam, and every amalgam is fine.
In terms of the sets $\calD_g$ defined in~\eqref{eq.def of D_g},
being fine equates to $\calD_g$ being empty when $g\in\calH$, and having equal cardinality otherwise.
When $\calD$ is an amalgam, we further have that each $\calD_g$ with $g\notin\calH$ is a difference set for $\calH$.
When $\calD$ is a composite, we even further have that any two such $\calD_g$ are translates.

Properly shifted, the complements of ``half" of all Singer difference sets are fine, and all of these are composite.
Meanwhile, the complement of every twin prime power difference set is fine,
and these are only amalgams when $Q\equiv3\bmod 4$, and only composite when $Q=3$.
No McFarland difference set is an amalgam.
Overall, we see that there are an infinite number of composite difference sets,
as well as an infinite number of fine difference sets that are not amalgams.
Moreover, there is a family of amalgams that are not composite, but whether or not this family is infinite depends on a form of the twin prime conjecture.

Every fine difference set $\calD$ yields an ECTFF in a natural way,
and this ECTFF is moreover an EITFF if and only if $\calD$ is an amalgam.
When $\calD$ is moreover composite, the corresponding orthogonal projection operators behave even more nicely than usual, satisfying~\eqref{eq.composite projections}.

Every composite difference set yields a simplicial RDS.
Moreover, every amalgam and every simplicial RDS yields a circulant conference matrix.
These constructions are two distinct generalizations of a common construction that applies to composite difference sets,
with each generalization producing examples that the other does not:
for any prime power $Q$, a simplicial RDS yields a circulant conference matrix of size $Q+1$; when $Q$ and $Q+2$ are twin prime powers with $Q\equiv 3\bmod 4$, the corresponding amalgam yields a circulant conference matrix of size $Q+2$.

%%%%%%%%%%%%%%%%%%%%%%%%%%%%%%%%%%%%%%%%%%%%%%%%%%%%%%%%%%%%%%%%
\section*{Acknowledgments}
We thank the three anonymous reviewers as well as Profs.\ Dustin Mixon and John Jasper for their many helpful suggestions.
The views expressed in this article are those of the authors and do not reflect the official policy or position of the United States Air Force, Department of Defense, or the U.S.~Government.

\appendix

\section{Alternate proofs of some parts of Theorem~\ref{thm.fine} }

For a combinatorial proof of some parts of Theorem~\ref{thm.fine},
recall that when $\calD$ is a difference set for $\calG$,
every nonzero member of $\calG$ appears exactly $\Lambda=\frac{D(D-1)}{G-1}=D-\frac{D^2}{S^2}$ times in its difference table.
As such, exactly $(H-1)\Lambda$ of these differences are nontrivial members of $\calH$.
Moreover, any given $g,g'\in\calG$ have the property that $g-g'\in\calH$ if and only if $g,g'$ lie in a common coset of $\calH$.
In particular,
partitioning $\calD$ as
$\calD=\sqcup_{g\in\calG/\calH}(g+\calD_g)$ leads to a corresponding partition of the nonzero $\calH$-valued entries of the difference table of $\calD$:
\begin{equation*}
\set{(d,d')\in\calD\times\calD: 0\neq d-d'\in\calH}
=\bigsqcup_{g\in\calG/\calH}\set{(d,d')\in(g+\calD_g)\times(g+\calD_g): d\neq d'}.
\end{equation*}
Counting these sets gives
$(H-1)\Lambda=\sum_{g\in\calG/\calH}D_g(D_g-1)=-D+\sqcup_{g\in\calG/\calH}D_g^2$
where $D_g:=\#(\calD_g)$.
Moreover, since $\calD$ is disjoint from $\calH$ we have that $D_g=0$ for the unique coset representative $g$ that lies in $\calH$, that is, such that $\overline{g}=\overline{0}$.
Overall, we have \smash{$\frac GH-1$} nonnegative integers
$\set{D_g}_{g\in\calG/\calH,\overline{g}\neq\overline{0}}$
such that $\sum_{\overline{g}\neq\overline{0}}D_g=D$ and
\smash{$\sum_{\overline{g}\neq\overline{0}}D_g^2=D+(H-1)\Lambda=(D-\Lambda)+H\Lambda$}.
Applying the Cauchy-Schwarz inequality to this sequence thus gives
\begin{equation}
\label{eq.pf of fine 2}
D^2\leq(\tfrac{G}{H}-1)[(D-\Lambda)+H\Lambda],
\end{equation}
where equality holds if and only if $D_g$ is constant over all $g\notin\calH$.
Now recall that \smash{$\Lambda=D-\frac{D^2}{S^2}$} where \smash{$S^2=\frac{D(G-1)}{(G-D)}$} and so
\smash{$\frac{GS^2\Lambda}{D^2}=G(\frac{S^2}{D}-1)=S^2-1$}.
Multiplying~\eqref{eq.pf of fine 2} by \smash{$\frac{GS^2}{HD^2}$} thus gives
\begin{equation*}
\tfrac{GS^2}{H}
\leq\tfrac{GS^2}{HD^2}(\tfrac{G}{H}-1)(\tfrac{D^2}{S^2}+H\Lambda)
=(\tfrac{G}{H}-1)(\tfrac{G}{H}+\tfrac{GS^2\Lambda}{D^2})
=(\tfrac{G}{H}-1)(\tfrac{G}{H}+S^2-1),
\end{equation*}
that is, $S^2\leq(\frac{G}{H}-1)^2$, namely the claim in Theorem~\ref{thm.fine}
that \smash{$H\leq\frac{G}{S+1}$}.
Moreover, when condition (i) of Theorem~\ref{thm.fine} holds, reversing the above argument gives equality in~\eqref{eq.pf of fine 2},
meaning \smash{$\set{D_g}_{g\in\calG/\calH,\overline{g}\neq\overline{0}}$} consists of $S$ equal numbers that sum to $D$, implying condition (iii).

For yet another alternative proof of one conclusion of Theorem~\ref{thm.fine}, note that if $\calD$ is disjoint from a subgroup $\calH$ of $\calG$ of order \smash{$H>\frac{G}{S+1}$},
then~\eqref{eq.simplex sum} implies that $\set{\bfphi_\gamma}_{\gamma\in\calH^\perp}$ is a linearly dependent subsequence of the harmonic ETF \smash{$\set{\bfphi_\gamma}_{\gamma\in\hat{\calG}}$} that consists of fewer than $S+1$ vectors; since \smash{$\set{\bfphi_\gamma}_{\gamma\in\hat{\calG}}$} has coherence $\frac1S$,
this violates a fact from compressed sensing known as the \textit{spark bound}~\cite{FickusJKM18}.

\end{document}